\documentclass[9pt, oneside]{amsart}   	
\usepackage{geometry}                		
\usepackage{graphicx}				
	
\usepackage{amssymb}
\usepackage{amsmath}
\usepackage{amsfonts}
\usepackage{amsthm}
\usepackage{amscd}
\usepackage{esint}
\usepackage{mathrsfs}

\usepackage{color}
\usepackage[usenames,dvipsnames,svgnames,table]{xcolor}

\usepackage[utf8]{inputenc}
\usepackage[parfill]{parskip}
\usepackage[colorlinks=true, pdfstartview=FitV,linkcolor=ForestGreen,citecolor=ForestGreen, urlcolor=black]{hyperref}

\usepackage{enumerate}
\usepackage{cite}
\usepackage{bbm}
\usepackage{comment}
\usepackage{url} 

\usepackage{tikz}

\newcommand{\dd}{\, {\rm d}}

\newcommand{\abs}[1]{\left\vert#1\right\vert}
\newcommand{\norm}[1]{\left\Vert#1\right\Vert}  

\newcommand{\R}{\ensuremath{{\mathbb R}}}

\newcommand{\beq}{\begin{equation}}
\newcommand{\eeq}{\end{equation}}
\newcommand{\beqs}{\begin{equation*}}
\newcommand{\eeqs}{\end{equation*}}
\newcommand{\bal}{\begin{equation}\begin{aligned}}
\newcommand{\eal}{\end{aligned}\end{equation}}
\newcommand{\bals}{\begin{equation*}\begin{aligned}}
\newcommand{\eals}{\end{aligned}\end{equation*}}

\newcounter{num} \numberwithin{num}{section}

\newtheorem{theorem}[num]{Theorem}
\newtheorem{proposition}[num]{Proposition}
\newtheorem{lemma}[num]{Lemma}

\theoremstyle{definition}
\newtheorem{definition}[num]{Definition}

\theoremstyle{remark}
\newtheorem{remark}[num]{Remark}

\numberwithin{equation}{section}

\title{Semi-local behaviour of non-local hypoelliptic equations: Boltzmann}
\author{Amélie Loher}
\date{Dec 16, 2024}

\address[Amélie Loher]{DPMMS, University of Cambridge, Wilberforce road, Cambridge CB3 0WA, UK}
\email{ajl221@cam.ac.uk}

\begin{document}

\begin{abstract}
The purpose of this note is to demonstrate the announced result in \cite{AL-ls} by filling the gap in the proof sketch. We prove the semi-local Strong Harnack inequality for the Boltzmann equation for moderately soft potentials without cutoff assumption. The non-local operator in the Boltzmann equation is in non-divergence form, 
and thus the method developed in \cite{AL-div} does not apply. However, we exploit that the Boltzmann equation is on average in divergence form, and we show that the non-divergent part of the collision operator is of lower order in a suitable sense, which proves to be sufficient to deduce the Strong Harnack inequality.
Consequentially, we derive upper and lower bounds on the fundamental solution of the linearised Boltzmann equation. 
\end{abstract}


\maketitle
\tableofcontents

\section{Introduction}

The aim of this article is to establish the Strong Harnack inequality for the Boltzmann equation for moderately soft potentials without cutoff assumption. 
The Boltzmann equation models the dynamics of a dilute gas. It is given by
\begin{equation}\label{eq:boltzmann}
	\partial_t f + v \cdot \nabla_x f = Q(f, f),
\end{equation}
where $f: \R \times \R^d \times \R^d \to [0, \infty)$ encodes the density of the gas particles, which at any given time $t \in \R$ have location $x \in \R^d$ and velocity $v \in \R^d$. The right hand side $Q$ denotes the Boltzmann collision operator, whose explicit form is given by
\begin{equation}\label{eq:Q}
	Q(f, f) = \int_{\R^d} \int_{\mathbb S^{d-1}} \left[f(w_*)f(w) - f(v_*)f(v)\right]B(\abs{v-v_*}, \cos \theta) \dd \sigma \dd v_*,
\end{equation}
where $v, v_*$ are the post-collisional velocities, and $w, w_*$ the pre-collisional velocities, so that
\beqs
	w = \frac{v + v_*}{2} +\frac{\abs{v - v_*}}{2}\sigma, \quad w_* = \frac{v + v_*}{2} -\frac{\abs{v - v_*}}{2}\sigma.
\eeqs	
The rate of change in velocities is determined through the cross-section $B$, which reads
\beq\label{eq:B}
	B(r, \cos\theta) = r^\gamma b(\cos\theta), \quad b(\cos\theta) \approx \abs{\sin\left(\frac{\theta}{2}\right)}^{-(d-1)-2s},
\eeq
with parameters $\gamma \in (-d, 1]$ and $s \in (0, 1)$, and where $\cos \theta$ is defined as 
\beqs
	\cos \theta := \frac{v - v_*}{\abs{v-v_*}} \cdot\sigma, \quad \sin\left(\frac{\theta}{2}\right) := \frac{w -v}{\abs{w-v}}\cdot \sigma.
\eeqs
We restrict our analysis to moderately soft potentials, which means we consider the range of parameters $0 \leq \gamma + 2s \leq 2$ throughout this article, and moreover, we are interested in long-range interactions, which occur when we do not cut off the singularity in the kernel at sufficiently small deviation angles $\theta$.

\subsection{Boltzmann collision operator}

The Boltzmann equation is a kinetic integro-differential equation with a bounded, non-negative source term. To see this, we split the collision operator $Q$ into two parts 
\beqs
	Q(f, f) = Q_1(f, f) + Q_2(f, f), 
\eeqs
where 
\bals
	Q_1(f, f) = \int_{\R^d} \int_{\mathbb S^{d-1}} f(w_*)\left[f(w) - f(v)\right]B(\abs{v-v_*}, \cos \theta) \dd \sigma \dd v_*,
\eals	
and
\bals
	Q_2(f, f) = f(v)\left(\int_{\R^d} \int_{\mathbb S^{d-1}} \left[f(w_*) - f(v_*)\right]B(\abs{v-v_*}, \cos \theta) \dd \sigma \dd v_*\right).
\eals	
Then, as shown in \cite[Lemma 4.1]{IS59}, $Q_1$ can be rewritten as
\beqs
	Q_1(f, f) =  \int_{\R^d} \left[f(t, x, w) -f(t, x, v)\right] K_f(t, x, v, w) \dd w,
\eeqs
where the kernel $K_f$ depends implicitly on the solution $f$ and is given by
\beq\label{eq:boltzmann_kernel}
	K_f(v, w) = 2^{d-1} \abs{w - v}^{-1}\int_{w' \perp w - v} f(v + w') B(r, \cos \theta) r^{-d+2} \dd w',
\eeq
with $r^2 = \abs{w-v}^2 + \abs{w'}^2 = \abs{v-v_*}^2$ and $\cos\theta = \frac{w' - (v-w)}{\abs{w' - (v-w)}}\cdot\frac{w' +(w-v)}{\abs{w' + (w-v)}}$, $w_* = v + w'$, and $v_* = w + w'$. The proof is based on a change of variables. In particular, $Q_1$ is a non-local operator with a non-negative kernel $K_f$. 

Moreover, the term $Q_2$ can be viewed as a non-negative source term due to cancellation effects.
It is well known in the kinetic literature, as stated for example in \cite[Lemma 5.1]{IS59}, that $Q_2$ is of the form
\beqs
	Q_2(f, f) =  f(v)\left(\int_{\R^d}f(v-w') \tilde B(\abs{w'}) \dd w'\right),
\eeqs
where 
\beq\label{eq:tilde-B}
	\tilde B(r) := \int_{\mathbb S^{d-1}} \left(\frac{2^{\frac{d}{2}}}{(1-\sigma \cdot e)^{\frac{d}{2}}} B\left(\frac{\sqrt r}{\sqrt{1-\sigma \cdot e}}, \cos \theta\right) - B(r, \cos\theta)\right) \dd \sigma > 0.
\eeq
Here $e$ is any unit vector, and $\cos \theta = \sigma \cdot e$.

\subsubsection{Integro-differential equation with non-negative source}

We have therefore derived the integro-differential form of the Boltzmann equation with a non-negative source term: if $f$ solves \eqref{eq:boltzmann}, then 
\beq\label{eq:boltzmann-id}
	\partial_t f(t, x, v) + v \cdot \nabla_x f(t, x, v) = \int_{\R^d} \big(f(t, x, w)- f(t, x, v)\big) K_f(t, x, v, w) \dd w + \Lambda_f f(t, x, v),
\eeq
where $K_f$ is given by \eqref{eq:boltzmann_kernel}, and $\Lambda_f$ is determined by
\beq\label{eq:Lambda-f}
	\Lambda_f := \int_{\R^d}f(v-w') \tilde B(\abs{w'}) \dd w' \geq 0,
\eeq
where $\tilde B$ is defined in \eqref{eq:tilde-B}. 

\subsubsection{Integro-differential equation in divergence form on average}

There is, however, also another way we can think of the Boltzmann equation, which proves to be useful for the purpose of this article. If we reconsider the term $Q_2$, then we realise that the change of variables according to \cite[Lemma A.9]{IS} together with \eqref{eq:boltzmann_kernel} yields
\bal\label{eq:Q2}
	Q_2(f, f) &=  f(v)\left(\int_{\R^d} \int_{\mathbb S^{d-1}} \big(f(w_*) - f(v_*)\big)B(\abs{v-v_*}, \cos \theta) \dd \sigma \dd v_*\right)\\
	&= f(v)\left(2^{d-1} \int_{\R^d} \int_{w' \perp w - v} \abs{w-v}^{-1}  B(r, \cos\theta) r^{-d+2}  \big(f(v+w') -f(w+w') \big) \dd w' \dd w\right)\\
	&= f(v)\left(\int_{\R^d}\big(K_f(v, w)-K_f(w, v)\big) \dd w\right)\\
\eal
Thus any solution of \eqref{eq:boltzmann} satisfies
\bal\label{eq:boltzmann-div}
	\partial_t f(t, x, v) + v \cdot \nabla_x f(t, x, v) &= \int_{\R^d} \big(f(t, x, w)- f(t, x, v)\big) K_f(t, x, v, w) \dd w \\
	&\quad +f(t,x,v)\left(\int_{\R^d}\big(K_f(t, x, v, w)-K_f(t, x, w, v)\big) \dd w\right)\\
	&=\frac{1}{2} \int_{\R^d} \big(f(t, x, w)- f(t, x, v)\big) \big(K_f(t, x, v, w) + K_f(t, x, w, v) \big)\dd w \\
	&\quad +\frac{1}{2}\int_{\R^d} \big(f(t, x, w)+ f(t, x, v)\big)\big(K_f(t, x, v, w)-K_f(t, x, w, v)\big) \dd w.
\eal
The strength of this reformulation \eqref{eq:boltzmann-div} is the emphasis on the fact that \textit{on average, the Boltzmann equation is an equation in divergence form}.

\subsection{Conditional regime}

In either formulation, \eqref{eq:boltzmann-id} and \eqref{eq:boltzmann-div}, we see that the coefficients $K_f$ \eqref{eq:boltzmann_kernel} and $\Lambda_f$ \eqref{eq:Lambda-f} depend implicitly on the solution, stressing the non-linear character of the equation. Without any further assumptions on the solution itself, it seems out of reach to treat the non-linearity in the equation. In order to make quantitative statements on the coefficients, we need to work in a conditional regime \cite{ISglobal}. We thus assume throughout this article that the solutions we consider satisfy the following hydrodynamic bounds: for any $(t, x) \in \R_+ \times \R^d$, there exists $m_0, M_0, E_0, H_0$, such that
\bal\label{eq:hydro}
	0 < m_0 \leq &\int_{\R^d} f(t,x,v) \dd v \leq M_0,\\
	&\int_{\R^d} f(t, x, v) \abs{v}^2 \dd v \leq E_0, \\
	&\int_{\R^d} f(t, x, v) \log f(t, x, v) \dd v \leq H_0.
\eal
In other words, the mass, the energy and the entropy are bounded. Then, Silvestre showed \cite[Theorem 1.2]{IS59}, that there is a constant $C_0 > 0$ depending only on $m_0, M_0, E_0, H_0$ such that 
\beqs
	\norm{f}_{L^\infty(\R_+\times \R^d \times \R^d)} \leq C_0. 
\eeqs
This conditional a priori bound on the solution enables us to quantify an ellipticity class on the coefficients $K_f$ and $\Lambda_f$, and thus the non-linearity is treated by absorbing it in the macroscopic quantities. Let us remark that this conditional regime, even though it involves physically meaningful quantities, it is in our opinion too much to ask for it to be mathematically justifiable in the space-inhomogeneous case, since these are pointwise assumptions in the space variable $x$.

\subsection{Main result}

In this conditional regime, a series of works by Imbert-Silvestre(-Mouhot) \cite{IS, ISglobal, ISschauder, IMS} showed first that any solution of \eqref{eq:boltzmann} satisfies the Weak Harnack inequality, and is thus Hölder continuous \cite{IS}, then that the regularity of the kernel $K$ can be transferred onto the solution in the sense of Schauder estimates \cite{ISschauder}, and finally that these local results can be globalised via a change of variables \cite[Section 5]{ISglobal} due to the decay properties of the solution \cite{IMS}, so that eventually the Schauder estimates can be bootstrapped to obtain smooth solutions \cite{ISglobal}. Later a constructive proof of the Weak Harnack inequality and the Schauder estimates appeared in \cite{AL, AL-schauder}. One of the difficulties posed by \eqref{eq:boltzmann} is its non-local operator, which means that the behaviour of the solution inside a given domain is affected by the values attained in the whole velocity space. In contrast to equations with a local diffusion operator, we cannot deduce from these Hölder estimates a bound on the local supremum of the solution in terms of the local infimum. Such a bound is known as the Strong Harnack inequality. The problem is that the behaviour of the tail is encoded in the constant appearing in the bound. Even though it is known that a local Strong Harnack inequality holds for parabolic non-local equations \cite{chen-kumagai}, it turns out to fail in the kinetic case \cite{KW-failedH}. The combination of the kinetic transport operator with a fractional diffusion affects the local supremum. It was crucial for their counterexample that the equation was satisfied in a bounded velocity domain. We showed in \cite{AL-div}, however, that a local Harnack inequality without tail terms holds for \textit{global} solutions to non-local kinetic equations in divergence form, that is we need the equation to be satisfied for velocities in the whole space. 
Here we show how to obtain the Strong Harnack inequality for global solutions to the Boltzmann equation. We think assuming the equation to be satisfied globally in $v$ is reasonable from a physical point of view, since velocities of gas particles are  vector fields defined on the whole space. 
\begin{theorem}[Strong Harnack for Boltzmann]\label{thm:sh-be}
Let $T > 0$. 
Suppose $0 \leq f$ satisfies \eqref{eq:hydro}. 
Let $f$ solve \eqref{eq:boltzmann} in $(0,T) \times \R^d \times \R^d$ in the sense of Definition \ref{def:sol} with a cross section $B$ given by \eqref{eq:B} with $0 \leq \gamma +2s \leq 2$, and with initial data $f_0(x, v) = f(0, x, v)$, which in case that $\gamma \leq 0$ is supposed to admit for any $q > 0$ a constant $N_q > 0$ such that $f_0(x, v)\leq N_q(1+\abs{v})^{-q}$ for all $(x, v) \in \R^d \times \R^d$.
Then for any $0 < r_0 <  \frac{1}{6}$, there is $C > 0$ such that
\beqs
	\sup_{(\tau_0, \tau_1) \times Q^t_{\frac{r_0}{4}}} f \leq C\inf_{(\tau_2, \tau_3) \times Q^t_{\frac{r_0}{4}}} f,
\eeqs
where $0 < \tau_0 < \tau_1 < \tau_2 < \tau_3$ such that $\tau_1 - \tau_0 = \tau_2 - \tau_3 = \big(\frac{r_0}{4}\big)^{2s}$ and $\tau_2 - \tau_1 \geq r_0^{2s}$, and where $Q^t_{\frac{r_0}{4}} := \big\{(x, v) \in\R^{2d} : \abs{x} < \big(\frac{r_0}{4}\big)^{1+2s}, \abs{v} < \frac{r_0}{4}\big\}$ is the  kinetic cylinder sliced in time.
The constant $C$ depends only on $s, d, \gamma, m_0, M_0, E_0, H_0$, and in case that $\gamma \leq 0$, it depends additionally on $N_q$.
\end{theorem}
This is the first time a Strong Harnack inequality for the Boltzmann equation has been derived. Moreover, we can deduce consequences on bounds on the fundamental solution, similar to \cite[Theorem 1.2 and Theorem 1.3]{AL-div}. We refer the reader to Section \ref{sec:bounds-fundsol}.

\bigskip
\noindent {\bf Acknowledgements.} 
We are extremely grateful for the incredible support we get from Clément Mouhot, and for his confidence in our work that endures long after ours vanished.

\section{Setup}
We first describe the notion of solutions that we work with.
\subsection{Notion of solutions}
Since we require further regularity properties of the solution to deduce that the skew-symmetric part of $K_f$ is of lower order in the sense of \eqref{eq:abs-K-skew}, we need to work with the following notion of solutions.
\begin{definition}[Solutions]\label{def:sol}
We say that $f: \big(0, T\big) \times \R^d \times \R^d \to \R$ is a solution of \eqref{eq:boltzmann}, if
\begin{enumerate}
	\item $f$ is non-negative everywhere, 
	\item $f$ solves \eqref{eq:boltzmann} classically for every $(t, x, v)$,
	\item $f \in C^\infty$ in time, space, and velocity,
	\item For each value of $(t, x)$ the function $f$ decays rapidly as $\abs{v} \to \infty$, that is for any $q > 0$ 
		\beqs
			\lim_{\abs{v} \to \infty} \frac{f(t, x, v)}{(1+\abs{v})^q}  = 0
		\eeqs
		locally uniformly in $(t, x)$.
	\item $f$ is periodic in $x$, 
	\item $f$ conserves mass:
	\[
		\forall 0 \leq t_1 <  t_2 \leq T:\iint_{\R^{2d}}f(t_2) \dd x \dd v \leq \iint_{\R^{2d}}f(t_1) \dd x \dd v.
	\]
\end{enumerate}
\end{definition}
This notion of solutions was also used by Imbert-Silvestre in \cite{ISglobal}. 

For any solution in the sense of Definition \ref{def:sol} satisfying the assumption on the macroscopic quantities \eqref{eq:hydro}, we can quantify a notion of ellipticity on the coefficients.
\subsection{Ellipticity class}
For any $(t, x) \in \R_+ \times \R^d$, we can show that there exists $\lambda_0, \Lambda_0, \mu_0 > 0$ depending on $\gamma, s, d, m_0, M_0, E_0, H_0$ such that $K_f$ given in \eqref{eq:boltzmann_kernel} and $\Lambda_f$ given in \eqref{eq:Lambda-f} satisfy the following statements. 

We start by noting that $\Lambda_f$ given in \eqref{eq:Lambda-f} is non-negative, and can be bounded by a constant depending only on $\gamma, s, d$ and the macroscopic bounds \eqref{eq:hydro}, that is
\beq\label{eq:cancellation}
	0 \leq \Lambda_f \leq \Lambda_0.
\eeq
Concerning $K_f$, first, we assume a weak form of coercivity: for any $\varphi \in C_c^\infty(\R^d)$ there holds
\beq
	\int_{\R^d}\int_{\R^d} \big(\varphi(v)-\varphi(w)\big)^2K_f(v, w) \dd w\dd v \geq	\lambda_0 \int_{\R^d}\int_{\R^d}  \frac{\abs{\varphi(v)-\varphi(w)}^2}{\abs{v-w}^{d+2s}} \dd v\dd w.
\label{eq:coercivity}
\eeq
Second, we assume an upper bound on average: for any $r > 0$
\beq
	\forall v \in \R^d \quad \int_{B_r(v)}K_f(v,w)\abs{v -w}^2\dd w \leq \Lambda_0 r^{2-2s}.
\label{eq:upperbound}
\eeq
Note, in particular, that this assumption is equivalent to
\beq\label{eq:upperbound-2}
	\forall v \in \R^d \quad \int_{\R^d \setminus B_r(v)}K_f(v,w)\abs{v -w}^2\dd w \leq \Lambda_0 r^{-2s}.
\eeq
Third, we discuss the symmetry of the kernel. Note that the kernel $K_f$ satisfies a pointwise symmetry of the form $K_f(v, v+w) = K_f(v, v-w)$.
This symmetry assumption gives rise to a non-local operator in non-divergence form. However, the kernel \textit{does not satisfy a pointwise divergence form symmetry}: $K_f(v, w) = K_f(w, v)$. This implies that the non-local operator with kernel $K_f$ is not self-adjoint. It does, however, satisfy a weak divergence form symmetry:
\beq
	\forall v \in \R^d \quad \Bigg\vert \textrm{PV} \int_{\R^d} \big(K_f(v, w) - K_f(w, v)\big)\dd w\Bigg\vert \leq \Lambda_0,
\label{eq:1.6}
\eeq
and if $s \geq \frac{1}{2}$ we assume that for all $r > 0$
\beq
	\forall v \in \R^d \quad \Bigg\lvert \textrm{PV} \int_{B_r(v)} (v -w)\big(K_f(v, w) - K_f(w, v)\big)\dd w\Bigg\rvert \leq \Lambda_0 r^{1-2s}.
\label{eq:1.7}
\eeq
These conditions mean that the anti-symmetric part is bounded.
It seems that conditions \eqref{eq:1.6} and \eqref{eq:1.7} are not sufficient to deduce our main theorem. We also need to assume that the anti-symmetric part is of lower order with respect to the symmetric part: for any $0 < r \leq 1$ there is $\alpha > s$ such that
\beq\label{eq:abs-K-skew}
	\sup_{v \in B_r} \int_{B_r} \frac{\abs{K_f(v, w) - K_f(w, v)}^2}{(K_f(v, w) + K_f(w, v))} \dd w \leq \mu_0r^{2(\alpha-s)}.
\eeq
This condition has been used to derive Hölder continuity of solutions to parabolic non-local equations by Kassmann-Weidner \cite{KW-a}.

The assumptions \eqref{eq:cancellation}-\eqref{eq:abs-K-skew} are satisfied for any solution $f$ of \eqref{eq:boltzmann} in $(0, T) \times \R^d \times \R^d$, such that \eqref{eq:hydro} holds, and in case that $\gamma \leq 0$, provided that for any $q > 0$ there is $N_q > 0$ such that $f_0(x, v) = f(0, x, v) \leq N_q(1+\abs{v})^{-q}$ for $(x, v) \in \R^d \times \R^d$. To be more precise, note that $\Lambda_0$ depends additionally on $\langle v\rangle$. Since we are interested in a statement on a bounded domain, we do not make this dependence explicit.  
For the justifications, we refer the reader to \cite[Lemma 5.2]{IS59} for \eqref{eq:cancellation}, to \cite[Theorem 1.2]{chaker-silvestre} for \eqref{eq:coercivity}, to \cite[Lemma 3.5, Lemma 3.6, Lemma 3.7]{IS} for \eqref{eq:upperbound}, \eqref{eq:1.6} and \eqref{eq:1.7}, respectively.
For a justification of the last assumption, we use further regularity properties of solutions to \eqref{eq:boltzmann}, derived by Imbert-Mouhot-Silvestre in \cite{ISglobal, IS, IMSfrench}.

\subsection{On the non-divergent part of the kernel}\label{subsec:antisym-low}
We consider a solution $f$ of \eqref{eq:boltzmann} in $(0, T) \times \R^d \times \R^d$, periodic in $x$, such that \eqref{eq:hydro} holds. If $\gamma \leq 0$, then we also assume that for any $q > 0$ there is $N_q > 0$ such that $f_0(x, v) = f(0, x, v) \leq N_q(1+\abs{v})^{-q}$ for $(x, v) \in \R^d \times \R^d$. Then for every $q > 0$ there is a constant $C_{s,q} > 0$ and $\alpha > s$ depending only on $s, d, \gamma, m_0, M_0, E_0, H_0$ and $N_q$ in case that $\gamma \leq 0$, such that for any $R > 0$, any $\tau_0 > 0$, any $(t, x) \in (\tau_0, T) \times B_{R^{1+2s}}$ and any $v, w \in \R^d$ with $\abs{v-w} < \min\big\{\frac{R}{2}, 1\big\}$ there holds
\beq\label{eq:holder-fast}
	\abs{f(v) - f(w)} \leq C_{s,q} \abs{v-w}^{\alpha} \big(1+\abs{v}\big)^{-q} \approx C_{s, q} \abs{v-w}^{\alpha} \big(1+\abs{w}\big)^{-q}. 
\eeq
An even stronger statement is derived in \cite[Corollary 7.8]{ISglobal}. In fact, the constants $C_{s, q}$ and $\alpha$ may depend on $N_p$ for some fixed $p > q$.

We now consider $(t, x, v) \in Q_R$ for $0 < R \leq 1$. We recall \cite[Corollary 4.2]{IS59}: for any $v, w\in \R^d$ there is a universal constant $C_K > 0$ such that
\beqs
	K_f(v,w) = C_K \abs{v-w}^{-(d+2s)}\int_{w' \perp v - w} f(v+w') \abs{w'}^{\gamma + 2s+1} \dd w'.
\eeqs 
Thus for $v \in B_R$ we take $q > d+2s+\gamma$ and use \eqref{eq:holder-fast}, so that
\begin{align*}
	&\abs{K_f(v, w) -K_f(w, v)} \\
	&\quad\leq C_K\abs{v-w}^{-(d+2s)} \abs{\int_{w' \perp v - w} f(v+w') \abs{w'}^{\gamma + 2s+1} \dd w' - \int_{w' \perp w - v} f(w+w') \abs{w'}^{\gamma + 2s+1} \dd w'}\\
	&\quad\leq C_K\abs{v-w}^{-(d+2s)} \int_{w' \perp v - w} \abs{f(v+w') -f(w+w')}\abs{w'}^{\gamma + 2s+1} \dd w'\\
	&\quad\leq C_KC_{s,q} \abs{v-w}^{\alpha-(d+2s)} \int_{w' \perp v - w}  \big(1+\abs{w'}\big)^{-q}\abs{w'}^{\gamma + 2s+1} \dd w'.
\end{align*}
We then combine this with \cite[Lemma 4.8]{IS59}, which shows that for each $v \in B_R$ there is a symmetric subset of the unit sphere $A = A(v) \subset \partial B_1$, with strictly positive measure, and such that for each $\frac{v-w}{\abs{v-w}} \in A$, the kernel satisfies a pointwise lower bound $K_f(v, w) \geq \lambda \abs{v-w}^{-(d+2s)}$ for some $\lambda$ depending only on the hydrodynamical quantities in \eqref{eq:hydro} and $R$. Indeed, Lemma 4.6 in \cite{IS59} implies that there is $\ell > 0$, $m > 0$ and $r > 0$ such that 
\[	
	\abs{\{v : f(v) > \ell \} \cap B_r} \geq m. 
\]	
Thus if we introduce $S:= \{v : f(v) > \ell \} \cap B_r$, then $f(v) \geq \ell \chi_S(v)$ and thus $K_f(v, w) \geq \ell K_{\chi_{S}}(v, w)$. Since $S \subset B_r$, we find on the one hand,
\[
	\int_{w' \perp v - w}\left( \chi_{S}(v+w') + \chi_S(w+w')\right) \abs{w'}^{\gamma + 2s+1} \dd w' \leq 2(R+r)^{\gamma + 2s + 1} \abs{B_r}.
\]
On the other hand, $\abs{S} \geq m$. 
In particular, this implies for any $v \in B_R$ that $K_f(v, w) \geq \lambda \abs{v-w}^{-(d+2s)}$ whenever $\frac{v-w}{\abs{v-w}} \in A$. 
We find
\begin{align*}
	\int_{B_R} &\frac{\abs{K_f(v, w) - K_f(w, v)}^2}{(K_f(v, w) + K_f(w, v))} \dd w =\int_{B_R} \frac{\abs{K_f(v, w) - K_f(w, v)}^2}{(K_f(v, w) + K_f(w, v))} (1-\chi_{K_f(v, w)=0}\chi_{K_f(w, v)=0})\dd w \\
	&\leq \frac{1}{\ell}C_K^2 C_{s, q}^2 \int_{B_R}   \frac{(1-\chi_{K_f(v, w)=0}\chi_{K_f(w, v)=0})}{\abs{v-w}^{d+2(s-\alpha)}} \frac{\int_{w' \perp v - w}  \big(1+\abs{w'}\big)^{-q}\abs{w'}^{\gamma + 2s+1} \dd w'}{\int_{w' \perp v - w}\left( \chi_{S}(v+w') + \chi_S(w+w')\right) \abs{w'}^{\gamma + 2s+1} \dd w'} \dd w\\
	&\leq \frac{C }{\ell}C_K^2C_{s, q}^2\int_0^R \rho^{2(\alpha-s)-1}\int_{\partial B_1} \left(  \frac{ (1-\chi_{K_f(v, w)=0}\chi_{K_f(w, v)=0})}{\int_{w' \perp \sigma}\left( \chi_{S}(v+w') + \chi_S(w+w')\right) \abs{w'}^{\gamma + 2s+1} \dd w'} \right) \dd \sigma \dd \rho \\
	&\leq C(\ell, \lambda, r, m) C_K^2C_{s, q}^2  R^{2(\alpha-s)}. 
\end{align*}
The last inequality follows from the fact that as $\sigma$ swipes $\partial B_1$, the hyperplane perpendicular to $\sigma$ covers the whole space $\R^d$, and using that $\abs{S} \geq m$, we see that at least one of the indicator functions in the denominator is positive; combined with the boundedness of $S$ we get a lower bound on the denominator.

\subsection{Notation}
We define the kinetic domains for time, space and velocity respecting the scaling of the equations. On the one hand, equation \eqref{eq:boltzmann} is scaling-invariant. Specifically, for any $r \in [0, 1]$ the scaled function $$f_r(t, x, v) = f(r^{2s}t, r^{1+2s}x, rv)$$ satisfies the Boltzmann equation with scaled coefficients. On the other hand, \eqref{eq:boltzmann} is invariant under Galilean transformations $$z \to z_0 \circ z = (t_0 + t, x_0+x + tv_0, v_0 + v)$$ with $z_0 = (t_0, x_0, v_0)\in \R^{1+2d}$. If $f$ is a solution of \eqref{eq:boltzmann}, then its Galilean transformation $f_{z_0}(z) = f(z_0 \circ z)$ solves \eqref{eq:boltzmann} with correspondingly translated coefficients. In view of these invariances we define kinetic cylinders 
\beqs
	Q_r(z_0) := \left\{(t, x, v) : -r^{2s} \leq t - t_0 \leq 0, \abs{v - v_0} < r, \abs{x - x_0 -(t-t_0)v_0} < r^{1+2s}\right\},
\eeqs
for $r > 0$ and $z_0 = (t_0, x_0, v_0) \in \R^{1+2d}$.

\section{Tail bound on upper level sets}

In this section we establish the non-local-to-local bound from Proposition 3.1 in \cite{AL-div} for the Boltzmann equation. It is in this section that the weak divergence form structure of \eqref{eq:boltzmann-div} proves useful. Just as in the divergence form case, we derive this non-local-to-local bound on upper level sets. To this end, we consider $l \in \R_+$. We check using \eqref{eq:boltzmann-div} that if $f$ solves \eqref{eq:boltzmann}, then $f-l$ satisfies
\bal\label{eq:boltzmann-div-l}
	 \partial_t (f-l) + v \cdot \nabla_x (f-l)   &= \frac{1}{2} \int_{\R^d} \big((f-l)(w) -(f-l)(v)\big) \big(K_f(v, w)+K_f(w, v)\big) \dd w   \\
	&\quad+ \frac{1}{2} \int_{\R^d}  \big((f-l)(v)+(f-l)(w)\big) \big(K_f( v, w)-K_f(w, v)\big)  \dd w\\
	&\quad+ l \int_{\R^d}  \big(K_f( v, w)-K_f(w, v)\big)  \dd w\\
	&\geq \frac{1}{2} \int_{\R^d} \big((f-l)(w) -(f-l)(v)\big) \big(K_f(v, w)+K_f(w, v)\big) \dd w   \\
	&\quad+ \frac{1}{2} \int_{\R^d}  \big((f-l)(v)+(f-l)(w)\big) \big(K_f( v, w)-K_f(w, v)\big)  \dd w.
\eal
The last inequality follows due to the classical cancellation from \eqref{eq:Q2} and \eqref{eq:tilde-B} provided that $l \geq 0$.

To derive a tail bound on upper level sets, we use the formulation in \eqref{eq:boltzmann-div-l}. The idea is to construct a concave test function that localises only on the upper level sets. Then, the weak divergence form structure proves to be sufficient to make the terms on the lower level sets vanish. Moreover, the cross terms end up with a sign.
\begin{proposition}[Non-local-to-local bound]\label{prop:tail-bound}
For given $R > 0$, $l \in \R_+$, and any non-negative super-solution $f$ of \eqref{eq:boltzmann-div} in $[0, T] \times \R^d \times \R^d$ with a kernel $K_f$ that satisfies \eqref{eq:upperbound}, \eqref{eq:1.6}, \eqref{eq:1.7}, \eqref{eq:abs-K-skew}, and such that $f$ conserves mass, we can find for any $\zeta \in \left(0, 1\right)$, a constant $C > 0$ depending on $\zeta, s, d, \Lambda_0, \mu_0$ such that 
\bal\label{eq:L1-tail}
	\int_{Q_{\frac{3R}{4}}}\int_{\R^d \setminus B_R} (f-l)_+(w) \chi_{f > l}(v) &K_f(v, w) \dd w\dd v \\
	&\leq CR^{-2s} \sup_{Q_R} (f-l)_+^{1-\zeta}\int_{Q_R} (f-l)_+^{\zeta} \dd z.
\eal
\end{proposition}
The proof of this proposition relies on the following preliminary estimates. 
\begin{lemma}\label{lem:prelim}
Let $R > 0$ and $v_0 \in \R^d$. Let $\eta \in C_c^\infty(\R^d)$ be some smooth function with support in $B_R$. Let $f: \R^d \to \R$ be non-negative, $l \geq 0$, $f_{l,\varepsilon} = (f -l)_+ +\varepsilon$ for some $\varepsilon > 0$ and $\zeta \in (0, 1)$. 
Then there holds:
\begin{enumerate}[i.] 
\item First, we have for any $v, w$ in $\{f(v) > l \}\cap \{f(w) > l\}$
\beq\label{eq:3.1}
	\big[f(v) - f(w)\big] \Big[f_{l,\varepsilon}^{-(1-\zeta)}(v)-  f_{l,\varepsilon}^{-(1-\zeta)}(w)\Big] \leq - \frac{4(1-\zeta)}{\zeta^2}\Big(f^{\frac{\zeta}{2}}_{l, \varepsilon}(v) - f^{\frac{\zeta}{2}}_{l,\varepsilon}(w)\Big)^2.
\eeq
\item Second, for any $v, w$ in $\{f(v) > l\} \cap\{ f(w) > l\}$
\bal\label{eq:3.2}
	\min\Big\{f_{l, \varepsilon}^{-(1-\zeta)}(v), f_{l, \varepsilon}^{-(1-\zeta)}(w)\Big\} &\abs{f(v) - f(w)}\leq \frac{2}{\zeta} \max\left\{f_{l, \varepsilon}^{\frac{\zeta}{2}}(v), f_{l, \varepsilon}^{\frac{\zeta}{2}}(w)\right\}\abs{f_{l, \varepsilon}^{\frac{\zeta}{2}}(v) - f_{l, \varepsilon}^{\frac{\zeta}{2}}(w)},\\
\eal
\item Third, for any $v, w$ in $\{f(v) > l \}\cap \{f(w) > l\}$
\bal\label{eq:3.3}
	\min\{(f-l)_+(v), (f-l)_+(w)\} &\abs{f_{l, \varepsilon}^{-(1-\zeta)}(v)- f_{l, \varepsilon}^{-(1-\zeta)}(w)} \\
	&\leq \frac{\zeta}{2(1-\zeta)} \abs{f_{l, \varepsilon}^{\frac{\zeta}{2}}(v) - f_{l, \varepsilon}^{\frac{\zeta}{2}}(w)}\min\left\{f_{l, \varepsilon}^{\frac{\zeta}{2}}(v), f_{l, \varepsilon}^{\frac{\zeta}{2}}(w)\right\}.
\eal
\item Fourth, 
\beq\label{eq:3.4}
	\frac{1}{2} \left\vert\left(\eta f_{l, \varepsilon}^{\frac{\zeta}{2}}\right)(v) -\left(\eta  f_{l, \varepsilon}^{\frac{\zeta}{2}}\right)(w)\right\vert^2 - \big(\eta(v) - \eta(w)\big)^2  f_{l, \varepsilon}^\zeta(w) \leq \min\big\{\eta^2(v), \eta^2(w)\big\} \abs{ f_{l, \varepsilon}^{\frac{\zeta}{2}}(v) -  f_{l, \varepsilon}^{\frac{\zeta}{2}}(w)}^2.
\eeq
and 
\bal\label{eq:3.5}
	\frac{1}{2}\max\big\{\eta^2(v), \eta^2(w)\big\} & \abs{  f_{l, \varepsilon}^{\frac{\zeta}{2}}(v) -  f_{l, \varepsilon}^{\frac{\zeta}{2}}(w) }^2 \leq   \abs{ \left(\eta  f_{l, \varepsilon}^{\frac{\zeta}{2}}\right)(v) - \left(\eta  f_{l, \varepsilon}^{\frac{\zeta}{2}}\right)(w) }^2 + \big(\eta(v) - \eta(w)\big)^2  f_{l, \varepsilon}^\zeta(v).
\eal
\end{enumerate}
\end{lemma}
\begin{proof}[Proof of Lemma \ref{lem:prelim}]
For a proof of \eqref{eq:3.2}, \eqref{eq:3.4}, \eqref{eq:3.5} we refer the reader to \cite[Lemma 3.3]{AL-div}. It remains to establish \eqref{eq:3.3}.
	To prove \eqref{eq:3.3}, if we use $$\tilde{\mathcal F}(\xi) := -\xi^{-\frac{\zeta}{2}\frac{1}{(1-\zeta)}}, \qquad \textrm{with} \qquad \tilde{\mathcal F}'(\xi) := \frac{\zeta}{2} \frac{1}{1-\zeta} \xi^{\frac{-(2-\zeta) }{2(1-\zeta)}},$$
	then	for all $v, w$ in $\{f(v) > l \cap f(w) > l\}$, we note by Cauchy's mean value theorem
	\bals
		\frac{\abs{f_{l, \varepsilon}^{-(1-\zeta)}(v) - f_{l, \varepsilon}^{-(1-\zeta)}(w)}}{\Big\vert f_{l, \varepsilon}^{\frac{\zeta}{2}}(v) - f_{l, \varepsilon}^{\frac{\zeta}{2}}(w)\Big\vert} &= \frac{\abs{f_{l, \varepsilon}^{-(1-\zeta)}(v) - f_{l, \varepsilon}^{-(1-\zeta)}(w)}}{\abs{\tilde{\mathcal F}\big(f_{l, \varepsilon}^{-(1-\zeta)}(w)\big) - \tilde{\mathcal F}\big(f_{l, \varepsilon}^{-(1-\zeta)}(v)\big)}} \\
		&\leq \max\Bigg\{\frac{1}{\tilde{\mathcal F}'\big(f_{l, \varepsilon}^{-(1-\zeta)}(v)\big)}, \frac{1}{\tilde{\mathcal F}'\big(f_{l, \varepsilon}^{-(1-\zeta)}(w)\big)}\Bigg\}\\
		&\leq \frac{2(1-\zeta)}{\zeta} \max\Bigg\{f_{l, \varepsilon}^{-\frac{(2-\zeta)}{2}}(v), f_{l, \varepsilon}^{-\frac{(2-\zeta)}{2}}(w)\big)\Bigg\}.
	\eals
	Multiplied by $\min\{(f-l)_+(v), (f-l)_+(w)\}$ yields \eqref{eq:3.3}. 
\end{proof}

\begin{proof}[Proof of Proposition \ref{prop:tail-bound}]
The proof idea follows \cite[Proposition 3.1]{AL-div}. We do a concavity estimate on level sets by testing the equation with the level set of the solution to some inverse power. The cross terms are signed, since the concavity of the test function preserves the super-solution structure. The weak divergence form structure implies, moreover, that if we localise only on upper level sets, then the terms on the lower level sets vanish.

We consider $0 \leq \psi$ given by
\beqs
	\psi_l(t, x, v)  := \eta^2(t, x, v) \chi_{f > l}(t, x, v) + \chi_{f < l}(t, x, v),
\eeqs
for $l \in \R_+$, and where $\eta \in C_c^\infty(\R^{2d+1})$ is such that $\eta = 1$ in $Q_{\frac{3R}{4}}$ and $\eta = 0$ outside $Q_{\frac{7R}{8}}$. 
With a slight abuse in notation, we will write below $\psi(v)$, $\eta(v)$ and $f(v)$ for $\psi(t, x, v)$, $\eta(t, x, v)$ and $f(t, x, v)$, respectively. 
In particular, for any $v, w \in \R^d$ there holds:
\begin{enumerate}[i.]
\item\label{itm:pos} If $f(v), f(w) > l$ then $\psi_l(v) = \eta^2(v)$ and $\psi_l(w) = \eta^2(w)$. 
\item\label{itm:neg} If $f(v), f(w) < l$ then $\psi_l(v) = \psi_l(w) = 1$. 
\item\label{itm:cross-v-w} If $f(v) > l > f(w)$ then $\psi_l(v) = \eta^2(v) \leq 1 = \psi_l(w)$. 
\item\label{itm:cross-w-v} If $f(w) > l > f(v)$ then $\psi_l(w) = \eta^2(w) \leq 1 = \psi_l(v)$.
\end{enumerate}
We test \eqref{eq:boltzmann-div-l} with $f_{l, \varepsilon}^{-(1-\zeta)} \psi_l$ for any $l \in \R_+$ and $\varepsilon > 0$. 

\textit{Step 1: Non-local operator.}

We first consider for fixed $(t, x)$ the right hand side of \eqref{eq:boltzmann-div-l} after testing with $f_{l, \varepsilon}^{-(1-\zeta)} \psi_l$:
\bals
	\mathcal E&:= \mathcal E\left(f-l, \psi f_{l, \varepsilon}^{-(1-\zeta)}\right)\\
	&= \int_{\R^d}\int_{\R^d} \left[(f-l)(v)K_f(w, v)-(f-l)(w)K_f(v, w)\right] \psi_l(v) f_{l,\varepsilon}^{-(1-\zeta)}(v)\dd w   \dd v\\
	&= \frac{1}{2} \int_{\R^d}\int_{\R^d} \left[(f-l)(v)K_f(w, v)-(f-l)(w)K_f(v, w)\right]  \left[\psi_l(v) f_{l,\varepsilon}^{-(1-\zeta)}(v) -\psi_l(w) f_{l,\varepsilon}^{-(1-\zeta)}(w) \right]\dd w   \dd v.
\eals
We then split $\mathcal E$ into three parts: we introduce
\bals
	&\chi_{up}(v, w) := \chi_{f(v) > l} \chi_{f(w) > l}, \\
	&\chi_{low}(v, w) := \chi_{f(v) < l} \chi_{f(w) < l},\\
	&\chi_{cross}(v, w) := \chi_{f(v) > l} \chi_{f(w) < l}+\chi_{f(v) < l} \chi_{f(w) > l},
\eals 
so that
\bal\label{eq:E-split}
	\mathcal E&=  \int_{\R^d}\int_{\R^d} \left[(f-l)(v)K_f(w, v) -(f-l)(w)K_f(v, w)\right] \psi_l(v) f_{l,\varepsilon}^{-(1-\zeta)}(v) \chi_{up}(v, w) \dd w   \dd v \\
	&\quad +  \int_{\R^d}\int_{\R^d} \left[(f-l)(v)K_f(w, v) -(f-l)(w)K_f(v, w)\right]  \psi_l(v) f_{l,\varepsilon}^{-(1-\zeta)}(v) \chi_{low}(v, w) \dd w   \dd v \\
	&\quad +\int_{\R^d}\int_{\R^d} \left[(f-l)(v)K_f(w, v) -(f-l)(w)K_f(v, w)\right] \psi_l(v) f_{l,\varepsilon}^{-(1-\zeta)}(v) \chi_{cross}(v, w) \dd w   \dd v \\
	&=: \mathcal E_{up} +\mathcal E_{low} +\mathcal E_{cross}. 
\eal

\textit{Step 1-i.: Non-locality on lower level sets.}

As a consequence of the weak divergence form structure of \eqref{eq:boltzmann-div}, we see that the choice of $\psi_l$ implies through observation \ref{itm:neg}, that
\bal\label{eq:E-neg}
	\mathcal E_{low}  &=  \frac{1}{2}\int_{\R^d}\int_{\R^d} \left[(f-l)(v)K_f(w, v) -(f-l)(w)K_f(v, w)\right]  \\
	&\qquad \qquad \qquad \times\left[\psi_l(v) f_{l,\varepsilon}^{-(1-\zeta)}(v)-\psi_l(v) f_{l,\varepsilon}^{-(1-\zeta)}(w)\right] \chi_{low}(v, w) \dd w   \dd v\\
	&= 0.
\eal

\textit{Step 1-ii.: Cross non-locality.}

Moreover, we claim that the cross term has a sign: 
\bal\label{eq:E-cross}
	\mathcal E_{cross} &=\frac{1}{2} \int_{\R^d}\int_{\R^d} \left[(f-l)(v)K_f(w, v) -(f-l)(w)K_f(v, w)\right] \\
	&\qquad \qquad \qquad \times \left[\psi_l(v) f_{l,\varepsilon}^{-(1-\zeta)}(v) -\psi_l(w) f_{l,\varepsilon}^{-(1-\zeta)}(w) \right]\chi_{cross}(v, w) \dd w   \dd v \\
	&=\frac{1}{2} \int_{\R^d}\int_{\R^d} \left[(f-l)(v)K_f(w, v) -(f-l)(w)K_f(v, w)\right] \\
	&\qquad \qquad \qquad \times \left[\eta^2(v)(f-l+\varepsilon)^{-(1-\zeta)}(v) - \varepsilon^{-(1-\zeta)} \right]\chi_{f(v) > l > f(w)} \dd w   \dd v \\
	&\quad+\frac{1}{2} \int_{\R^d}\int_{\R^d} \left[(f-l)(v)K_f(w, v) -(f-l)(w)K_f(v, w)\right] \\
	&\qquad \qquad \qquad \times \left[\varepsilon^{-(1-\zeta)} - \eta^2(w)(f-l+\varepsilon)^{-(1-\zeta)}(w) \right]\chi_{f(w) > l > f(v)} \dd w   \dd v \\
	&\leq 0.
\eal
The last inequality uses, for the first integrand, that we are restricted to the case that $f(v) -l > 0$ and $-(f(w) -l) > 0$, combined with the fact that $\eta^2(v) (f-l+\varepsilon)^{-(1-\zeta)}(v) \leq \varepsilon^{-(1-\zeta)}$, since $f(v) > l$ and $\eta^2(v) \leq 1$. Similarly, we argue for the second integrand, after noticing that $(f-l)(v) < 0$ and $-(f-l)(w) < 0$.

\textit{Step 1-iii.: Non-locality on upper level sets.}

For $\mathcal E_{up}$ we find for any $(t, x) \in [-R^{2s}, 0] \times B_{R^{1+2s}}$ due to \ref{itm:pos} 
\bal\label{eq:E-pos}
	\mathcal E_{up} &=  \int_{\R^d}\int_{\R^d} \left[(f-l)(v)K_f(w, v) -(f-l)(w)K_f(v, w)\right] \psi_l(v) f_{l,\varepsilon}^{-(1-\zeta)}(v) \chi_{up}(v, w) \dd w   \dd v \\
	&=\int_{B_R}\int_{B_R} \left[(f-l)(v)K_f(w, v) -(f-l)(w)K_f(v, w)\right] \psi_l(v) f_{l,\varepsilon}^{-(1-\zeta)}(v) \chi_{up}(v, w)  \dd w \dd v\\
	&\quad + \int_{B_R}\int_{\R^d \setminus B_R} \left[(f-l)(v)K_f(w, v) -(f-l)(w)K_f(v, w)\right] \psi_l(v) f_{l,\varepsilon}^{-(1-\zeta)}(v) \chi_{up}(v, w) \dd w \dd v\\
	&=: \mathcal E_{up}^{loc} + \mathcal E_{up}^{tail}. 
\eal
For $\mathcal E_{up}^{tail}$, we note that there is no singularity, due to the choice of $\eta$, which vanishes outside $B_{\frac{7R}{8}}$. Thus we find due to \eqref{eq:upperbound-2}
\bal\label{eq:E-pos-tail}
	\mathcal E_{up}^{tail} &= \int_{B_R}\int_{\R^d \setminus B_R}\left[(f-l)(v)K_f(w, v) -(f-l)(w)K_f(v, w)\right] \eta^2(v) f_{l,\varepsilon}^{-(1-\zeta)}(v) \chi_{f(v) > l} \chi_{f(w)> l} \dd w \dd v\\
	&= \int_{B_R}\int_{\R^d \setminus B_R} (f-l)_+(v) \eta^2(v) f_{l,\varepsilon}^{-(1-\zeta)}(v) K_f(w, v)\chi_{f(v) > l} \chi_{f(w)> l} \dd w \dd v\\
	&\quad - \int_{B_R}\int_{\R^d \setminus B_R}(f-l)_+(w) \eta^2(v) f_{l,\varepsilon}^{-(1-\zeta)}(v) K_f(v, w)\chi_{f(v) > l} \chi_{f(w)> l} \dd w \dd v\\
	&\leq \Lambda_0 \left(\frac{R}{8}\right)^{-2s} \int_{B_R} ((f-l)_++\varepsilon)^{\zeta}(v) \chi_{f > l}(v) \dd v\\
	&\quad- \int_{B_R}\int_{\R^d \setminus B_R}(f-l)_+(w) \eta^2(v) f_{l,\varepsilon}^{-(1-\zeta)}(v) K_f(v, w)\chi_{f(v) > l} \dd w \dd v.
\eal
In particular, the tail of $(f-l)_+$, which corresponds to the last term in the above inequality, has a good sign.

To bound the not-too-non-local non-locality of $\mathcal E_{up}$, we exploit the concavity of the test function again, due to which we get a signed coercive term, and we show that the terms that have no sign have a singularity of lower order, and thus can be bounded. Concretely, using that $1 = \chi_{f(w) > f(v)} + \chi_{f(v) > f(w)}$, the symmetry of the integrand, adding and subtracting $f_{l,\varepsilon}^{-(1-\zeta)}(w)\eta^2(v)$, and finally adding and subtracting $f_l(v)K_f(v, w)$ and $f_l(w)K_f(w, v)$,
we find
\bal\label{eq:E-up-loc}
	\mathcal E_{up}^{loc} &=\frac{1}{2}\int_{B_R}\int_{B_R} \left[(f-l)(v)K_f(w, v) -(f-l)(w)K_f(v, w)\right] \\
	&\qquad \qquad \qquad \times\left[\eta^2(v) f_{l,\varepsilon}^{-(1-\zeta)}(v)-\eta^2(w) f_{l,\varepsilon}^{-(1-\zeta)}(w)\right] \chi_{up}(v, w)  \dd w \dd v\\
	&=\int_{B_R}\int_{B_R} \left[(f-l)(v)K_f(w, v) -(f-l)(w)K_f(v, w)\right] \\
	&\qquad \qquad \qquad \times \left[\eta^2(v) f_{l,\varepsilon}^{-(1-\zeta)}(v)-\eta^2(w) f_{l,\varepsilon}^{-(1-\zeta)}(w)\right] \chi_{up}(v, w) \chi_{f(w) > f(v)} \dd w \dd v\\
	&=\int_{B_R}\int_{B_R} \left[(f-l)(v)K_f(w, v) -(f-l)(w)K_f(v, w)\right]\\
	&\qquad \qquad \qquad \times \left[f_{l,\varepsilon}^{-(1-\zeta)}(v)- f_{l,\varepsilon}^{-(1-\zeta)}(w)\right] \eta^2(v) \chi_{up}(v, w) \chi_{f(w) > f(v)} \dd w \dd v\\
	&\quad +\int_{B_R}\int_{B_R} \left[(f-l)(v)K_f(w, v) -(f-l)(w)K_f(v, w)\right]\\
	&\qquad \qquad \qquad \times \left[\eta^2(v)- \eta^2(w)\right] f_{l,\varepsilon}^{-(1-\zeta)}(w) \chi_{up}(v, w) \chi_{f(w) > f(v)} \dd w \dd v\\
	&=\int_{B_R}\int_{B_R} \left[f_l(v) -f_l(w)\right] \left[f_{l,\varepsilon}^{-(1-\zeta)}(v)- f_{l,\varepsilon}^{-(1-\zeta)}(w)\right] \eta^2(v)K_f(v, w) \chi_{up} \chi_{f(w) > f(v)} \dd w \dd v\\
	&\quad+\int_{B_R}\int_{B_R} \left[K_f(w,v) -K_f(v, w)\right] f_l(v) \left[f_{l,\varepsilon}^{-(1-\zeta)}(v)- f_{l,\varepsilon}^{-(1-\zeta)}(w)\right] \eta^2(v) \chi_{up} \chi_{f(w) > f(v)} \dd w \dd v\\
	&\quad +\int_{B_R}\int_{B_R} \left[f_l(v) -f_l(w)\right] \left[\eta^2(v)- \eta^2(w)\right] f_{l,\varepsilon}^{-(1-\zeta)}(w) K_f(w, v)\chi_{up}\chi_{f(w) > f(v)} \dd w \dd v\\
	&\quad +\int_{B_R}\int_{B_R} \left[K_f(w, v) -K_f(v, w)\right] \left[\eta^2(v)- \eta^2(w)\right] f_l(w)f_{l,\varepsilon}^{-(1-\zeta)}(w) \chi_{up}\chi_{f(w) > f(v)} \dd w \dd v\\
	&=: I_{concave}^{f} + I_{concave}^{K} +I_{cutoff}^{f}+ I_{cutoff}^{K},
\eal
where we denoted $f_l = f - l$.

For $I^f_{concave}$ we use \eqref{eq:3.1}, \eqref{eq:3.4}, and \eqref{eq:upperbound}
\bals
	I^f_{concave} &\leq - \frac{4(1-\zeta)}{\zeta^2} \int_{B_R}\int_{B_R} \Big(f^{\frac{\zeta}{2}}_{l, \varepsilon}(v) - f^{\frac{\zeta}{2}}_{l,\varepsilon}(w)\Big)^2\eta^2(v)K_f(v, w) \chi_{up}(v, w) \chi_{f(w) > f(v)} \dd w \dd v\\
		&\leq \frac{4(1-\zeta)}{\zeta^2} \int_{B_R}\int_{B_R}f^{\zeta}_{l, \varepsilon}(w) \big(\eta(v) - \eta(w)\big)^2 K_f(v, w) \dd w \dd v \\
		&\quad - \frac{2(1-\zeta)}{\zeta^2} \int_{B_R}\int_{B_R} \Big(\big(\eta f^{\frac{\zeta}{2}}_{l, \varepsilon}\big)(v) - \big(\eta f^{\frac{\zeta}{2}}_{l,\varepsilon}\big)(w)\Big)^2K_f(v, w) \chi_{up}(v, w) \chi_{f(w) > f(v)} \dd w \dd v.
\eals
Then we use the upper bound \eqref{eq:upperbound}. Moreover, in order to absorb the remaining terms by the signed term, we distinguish the symmetric and the anti-symmetric part by adding and subtracting $K_f(w, v)$, then using Young's inequality and \eqref{eq:abs-K-skew}, so that for $\delta$ sufficiently small
\bal\label{eq:I-concave-f}
	I^f_{concave} &\leq C R^{2-2s} \norm{D_v\eta}^2_{L^\infty} \int_{B_R}f^{\zeta}_{l, \varepsilon}(w)\dd w \\
		&\quad - \frac{(1-\zeta)}{\zeta^2} \int_{B_R}\int_{B_R} \Big(\big(\eta f^{\frac{\zeta}{2}}_{l, \varepsilon}\big)(v) - \big(\eta f^{\frac{\zeta}{2}}_{l,\varepsilon}\big)(w)\Big)^2(K_f(v, w)+K_f(w, v))  \chi_{f(w) > f(v) >l} \dd w \dd v\\
		&\quad + \frac{(1-\zeta)}{\zeta^2} \int_{B_R}\int_{B_R} \Big(\big(\eta f^{\frac{\zeta}{2}}_{l, \varepsilon}\big)(v) - \big(\eta f^{\frac{\zeta}{2}}_{l,\varepsilon}\big)(w)\Big)^2(K_f(w, v)-K_f(v, w)) \chi_{f(w) > f(v) >l}  \dd w \dd v\\
		&\leq C R^{-2s}  \int_{B_R}f^{\zeta}_{l, \varepsilon}(w)\dd w \\
		&\quad - \frac{(1-\zeta)}{\zeta^2} \int_{B_R}\int_{B_R} \Big(\big(\eta f^{\frac{\zeta}{2}}_{l, \varepsilon}\big)(v) - \big(\eta f^{\frac{\zeta}{2}}_{l,\varepsilon}\big)(w)\Big)^2(K_f(v, w)+K_f(w, v))  \chi_{f(w) > f(v) >l} \dd w \dd v\\
		&\quad +\delta \int_{B_R}\int_{B_R} \Big(\big(\eta f^{\frac{\zeta}{2}}_{l, \varepsilon}\big)(v) - \big(\eta f^{\frac{\zeta}{2}}_{l,\varepsilon}\big)(w)\Big)^2(K_f(w, v)+K_f(v, w))  \chi_{f(w) > f(v) >l} \dd w \dd v\\
		&\quad +C(\delta, \zeta) \int_{B_R}\int_{B_R} \Big(\big(\eta f^{\frac{\zeta}{2}}_{l, \varepsilon}\big)(v) - \big(\eta f^{\frac{\zeta}{2}}_{l,\varepsilon}\big)(w)\Big)^2
	\frac{\abs{K_f(w, v)-K_f(v, w)}^2}{(K_f(v, w) + K_f(w, v))}  \chi_{f(w) > f(v) >l}  \dd w \dd v\\
	&\leq C R^{-2s} \int_{B_R}f^{\zeta}_{l, \varepsilon}(w)\dd w \\
		&\quad - \frac{(1-\zeta)}{2\zeta^2} \int_{B_R}\int_{B_R} \Big(\big(\eta f^{\frac{\zeta}{2}}_{l, \varepsilon}\big)(v) - \big(\eta f^{\frac{\zeta}{2}}_{l,\varepsilon}\big)(w)\Big)^2(K_f(v, w)+K_f(w, v)) \chi_{f(w) > f(v) >l} \dd w \dd v\\
		&\quad +C(\delta, \zeta) \int_{B_R}\eta^2(v) f^{\zeta}_{l, \varepsilon}(v)  \sup_{v \in B_R} \int_{B_R} 	\frac{\abs{K_f(w, v)-K_f(v, w)}^2}{(K_f(v, w) + K_f(w, v))} \dd w \dd v\\
		&\quad +C(\delta, \zeta) \int_{B_R}\eta^2(w) f^{\zeta}_{l, \varepsilon}(w)  \sup_{w \in B_R} \int_{B_R}  \frac{\abs{K_f(w, v)-K_f(v, w)}^2}{(K_f(v, w) + K_f(w, v))} \dd v \dd w\\
		&\leq C R^{-2s} \int_{B_R}f^{\zeta}_{l, \varepsilon}(w)\dd w \\
		&\quad - \frac{(1-\zeta)}{2\zeta^2} \int_{B_R}\int_{B_R} \Big(\big(\eta f^{\frac{\zeta}{2}}_{l, \varepsilon}\big)(v) - \big(\eta f^{\frac{\zeta}{2}}_{l,\varepsilon}\big)(w)\Big)^2(K_f(v, w)+K_f(w, v)) \chi_{f(w) > f(v) >l}  \dd w \dd v.
\eal

To bound $I_{cutoff}^K$ we use $\abs{\eta^2(v)- \eta^2(w)}  \leq \max\{\eta(v), \eta(w)\}\abs{\eta(v)-\eta(w)}$, $(f-l)_+ \leq f_{l, \varepsilon}$, Young's inequality, \eqref{eq:abs-K-skew}, and \eqref{eq:upperbound} 
\bal\label{eq:I-cutoff-K}
	I_{cutoff}^K &= \int_{B_R}\int_{B_R} \left[K_f(w, v) -K_f(v, w)\right] \left[\eta^2(v)- \eta^2(w)\right] (f-l)(w)f_{l,\varepsilon}^{-(1-\zeta)}(w)  \chi_{f(w) > f(v) >l}  \dd w \dd v\\
	&\leq\int_{B_R}\int_{B_R} \frac{\abs{K_f(w, v) -K_f(v, w)}^2}{(K_f(v, w) +K_f(w, v))} \max\{\eta^2(v), \eta^2(w)\} f_{l,\varepsilon}^{\zeta}(w)  \chi_{f(w) > f(v) >l} \dd w\dd v\\
	&\quad+ \int_{B_R}\int_{B_R}(K_f(w, v)+K_f(v, w)) \abs{\eta(v)- \eta(w)}^2 f_{l,\varepsilon}^{\zeta}(w)  \chi_{f(w) > f(v) >l}  \dd w \dd v\\
	&\leq C(\mu_0) \int_{B_R} f_{l,\varepsilon}^{\zeta}(w) \dd w+C R^{2-2s} \norm{D\eta}_{L^\infty}^2 \int_{B_R} f_{l,\varepsilon}^{\zeta}(w)  \dd w.
\eal

For $I_{cutoff}^f$, we use \eqref{eq:3.2}, Young's inequality, \eqref{eq:3.5}, the upper bound \eqref{eq:upperbound}, we add and subtract $K_f(w, v)$, we use Young's inequality again, and \eqref{eq:abs-K-skew}, so that 
\bal\label{eq:I-cutoff-f}
	I_{cutoff}^f &= \int_{B_R}\int_{B_R} \left[(f-l)(v) -(f-l)(w)\right] \left[\eta^2(v)- \eta^2(w)\right] f_{l,\varepsilon}^{-(1-\zeta)}(w) K_f(w, v) \chi_{f(w) > f(v) >l} \dd w \dd v\\		
	&\leq \frac{2}{\zeta}\int_{B_R}\int_{B_R} \abs{f_{l, \varepsilon}^{\frac{\zeta}{2}}(v) - f_{l, \varepsilon}^{\frac{\zeta}{2}}(w)} f_{l, \varepsilon}^{\frac{\zeta}{2}}(w)\abs{\eta^2(v)- \eta^2(w)} K_f(w, v) \chi_{f(w) > f(v) >l}  \dd w \dd v\\	
	&\leq\frac{\delta}{2}\int_{B_R}\int_{B_R} \abs{f_{l, \varepsilon}^{\frac{\zeta}{2}}(v) - f_{l, \varepsilon}^{\frac{\zeta}{2}}(w)}^2 \max\{\eta^2(v), \eta^2(w)\} K_f(w, v) \chi_{f(w) > f(v) >l}  \dd w \dd v\\	
	&\quad +C(\delta) \int_{B_R}\int_{B_R}f_{l, \varepsilon}^{\zeta}(w)\abs{\eta(v)- \eta(w)}^2 K_f(w, v)\chi_{up}(v, w) \chi_{f(w) > f(v)} \dd w \dd v\\	
	&\leq\delta\int_{B_R}\int_{B_R} \Big(\big(\eta f^{\frac{\zeta}{2}}_{l, \varepsilon}\big)(v) - \big(\eta f^{\frac{\zeta}{2}}_{l,\varepsilon}\big)(w)\Big)^2  K_f(w, v) \chi_{f(w) > f(v) >l}  \dd w \dd v\\	
	&\quad +\delta\int_{B_R}\int_{B_R} f_{l, \varepsilon}^{\zeta}(w) (\eta(v) - \eta(w))^2 K_f(w, v) \chi_{f(w) > f(v) >l}  \dd w \dd v\\
	&\quad +C(\delta)\norm{D_v \eta}_{L^\infty}^2 R^{2-2s} \int_{B_R}f_{l, \varepsilon}^{\zeta}(w)\dd w\\
	&\leq\frac{\delta}{2}\int_{B_R}\int_{B_R}\Big(\big(\eta f^{\frac{\zeta}{2}}_{l, \varepsilon}\big)(v) - \big(\eta f^{\frac{\zeta}{2}}_{l,\varepsilon}\big)(w)\Big)^2  (K_f(v, w) + K_f(w, v)) \chi_{f(w) > f(v) >l}  \dd w \dd v\\	
	&\quad+\frac{\delta}{2}\int_{B_R}\int_{B_R} \Big(\big(\eta f^{\frac{\zeta}{2}}_{l, \varepsilon}\big)(v) - \big(\eta f^{\frac{\zeta}{2}}_{l,\varepsilon}\big)(w)\Big)^2 (K_f(w, v) -K_f(v, w)) \chi_{f(w) > f(v) >l} \dd w \dd v\\	
	&\quad +C(\delta) R^{-2s} \int_{B_R}f_{l, \varepsilon}^{\zeta}(w)\dd w\\
	&\leq\frac{\delta}{2}\int_{B_R}\int_{B_R}\Big(\big(\eta f^{\frac{\zeta}{2}}_{l, \varepsilon}\big)(v) - \big(\eta f^{\frac{\zeta}{2}}_{l,\varepsilon}\big)(w)\Big)^2(K_f(v, w) + K_f(w, v)) \chi_{f(w) > f(v) >l}  \dd w \dd v\\	
	&\quad +C(\delta)\int_{B_R}\int_{B_R} \Big(\big(\eta f^{\frac{\zeta}{2}}_{l, \varepsilon}\big)(v) - \big(\eta f^{\frac{\zeta}{2}}_{l,\varepsilon}\big)(w)\Big)^2  \frac{\abs{K_f(v, w) - K_f(w, v)}^2}{(K_f(v, w) + K_f(w, v))} \chi_{f(w) > f(v) >l} \dd w \dd v\\
	&\quad +\frac{\delta}{2} \int_{B_R}\int_{B_R}\Big(\big(\eta f^{\frac{\zeta}{2}}_{l, \varepsilon}\big)(v) - \big(\eta f^{\frac{\zeta}{2}}_{l,\varepsilon}\big)(w)\Big)^2 (K_f(v, w) + K_f(w, v))  \chi_{f(w) > f(v) >l} \dd w \dd v\\
	&\quad +C(\delta) R^{-2s} \int_{B_R}f_{l, \varepsilon}^{\zeta}(w)\dd w\\
	&\leq\delta\int_{B_R}\int_{B_R} \Big(\big(\eta f^{\frac{\zeta}{2}}_{l, \varepsilon}\big)(v) - \big(\eta f^{\frac{\zeta}{2}}_{l,\varepsilon}\big)(w)\Big)^2(K_f(v, w) + K_f(w, v)) \chi_{f(w) > f(v) >l}  \dd w \dd v\\	
	&\quad +C(\delta) R^{-2s} \int_{B_R}f_{l, \varepsilon}^{\zeta}(w)\dd w.
\eal

Finally, to bound $I_{concave}^K$, we use \eqref{eq:3.3}, Young's inequality, \eqref{eq:abs-K-skew}, \eqref{eq:3.5}, and the upper bound \eqref{eq:upperbound}
\bal\label{eq:I-concave-K}
	I_{concave}^K&= \int_{B_R}\int_{B_R} \left[K_f(w,v) -K_f(v, w)\right] \left[f_{l,\varepsilon}^{-(1-\zeta)}(v)- f_{l,\varepsilon}^{-(1-\zeta)}(w)\right] f_l(v)\eta^2(v)  \chi_{f(w) > f(v) >l} \dd w \dd v\\
	&\leq \frac{\zeta}{(1-\zeta)}\int_{B_R}\int_{B_R} \abs{K_f(w,v) -K_f(v, w)} \abs{f_{l, \varepsilon}^{\frac{\zeta}{2}}(v) - f_{l, \varepsilon}^{\frac{\zeta}{2}}(w)} f_{l, \varepsilon}^{\frac{\zeta}{2}}(v)\eta^2(v) \chi_{f(w) > f(v) >l} \dd w \dd v\\
	&\leq C(\delta)\int_{B_R}\int_{B_R} \frac{\abs{K_f(w,v) -K_f(v, w)}^2}{(K_f(v, w)+ K_f(w, v))}f_{l, \varepsilon}^{\zeta}(v)\eta^2(v) \chi_{f(w) > f(v) >l} \dd w \dd v\\
	&\quad + \frac{\delta}{2}\int_{B_R}\int_{B_R} (K_f(v, w)+ K_f(w, v))\abs{f_{l, \varepsilon}^{\frac{\zeta}{2}}(v) - f_{l, \varepsilon}^{\frac{\zeta}{2}}(w)}\eta^2(v) \chi_{f(w) > f(v) >l}  \dd w \dd v\\
	&\leq C\int_{B_R}f_{l, \varepsilon}^{\zeta}(v)  \dd v \\
	&\quad+ \delta\int_{B_R}\int_{B_R} (K_f(v, w)+ K_f(w, v))\Big(\big(\eta f^{\frac{\zeta}{2}}_{l, \varepsilon}\big)(v) - \big(\eta f^{\frac{\zeta}{2}}_{l,\varepsilon}\big)(w)\Big)^2 \chi_{f(w) > f(v) >l}  \dd w \dd v\\
	&\quad+\delta\int_{B_R}\int_{B_R} (K_f(v, w)+ K_f(w, v))\big(\eta(v) - \eta(w)\big)^2  f_{l, \varepsilon}^\zeta(v) \chi_{f(w) > f(v) >l}  \dd w \dd v\\
	&\leq C \int_{B_R} f_{l,\varepsilon}^{\zeta}(v) \dd v+ C R^{2-2s} \norm{D_v \eta}_{L^\infty}^2 \int_{B_R} f_{l,\varepsilon}^{\zeta}(w) \dd w\\
	&\quad + \delta\int_{B_R}\int_{B_R} (K_f(v, w)+ K_f(w, v))\Big(\big(\eta f^{\frac{\zeta}{2}}_{l, \varepsilon}\big)(v) - \big(\eta f^{\frac{\zeta}{2}}_{l,\varepsilon}\big)(w)\Big)^2 \chi_{f(w) > f(v) >l}  \dd w \dd v.
\eal

If we combine \eqref{eq:E-up-loc} with \eqref{eq:I-concave-f}, \eqref{eq:I-cutoff-K}, \eqref{eq:I-cutoff-f}, \eqref{eq:I-concave-K}, we find for $\delta$ sufficiently small
\bal\label{eq:E-pos-loc}
	\mathcal E_{up}^{loc} &\leq C R^{-2s} \int_{B_R}f^{\zeta}_{l, \varepsilon}(w)\dd w \\
	&\quad - \frac{(1-\zeta)}{2\zeta^2} \int_{B_R}\int_{B_R} \Big(\big(\eta f^{\frac{\zeta}{2}}_{l, \varepsilon}\big)(v) - \big(\eta f^{\frac{\zeta}{2}}_{l,\varepsilon}\big)(w)\Big)^2(K_f(v, w)+K_f(w, v)) \chi_{f(w) > f(v) > l} \dd w \dd v\\
	&\quad +2\delta\int_{B_R}\int_{B_R} \left(\big(\eta f_{l, \varepsilon}^{\frac{\zeta}{2}}\big)(v) - \big(\eta f_{l, \varepsilon}^{\frac{\zeta}{2}}\big)(w)\right)^2  (K_f(v, w) + K_f(w, v)) \chi_{f(w) > f(v) > l} \dd w \dd v\\	
	&\leq C R^{-2s} \int_{B_R}f^{\zeta}_{l, \varepsilon}(w)\dd w.
\eal
In particular, this implies together with \eqref{eq:E-pos} and \eqref{eq:E-pos-tail} 
\bal\label{eq:E-up}
	\mathcal E_{up} \leq C R^{-2s} \int_{B_R}f^{\zeta}_{l, \varepsilon}(w)\dd w -\int_{B_R}\int_{\R^d \setminus B_R}(f-l)_+(w) \eta^2(v) f_{l,\varepsilon}^{-(1-\zeta)}(v) K_f(v, w)\chi_{f(v) > l} \dd w \dd v.
\eal

\textit{Step 1-iv.: Non-locality.}

We conclude from \eqref{eq:E-split}, \eqref{eq:E-neg}, \eqref{eq:E-cross}, \eqref{eq:E-up}, 
\bal\label{eq:E}
	\mathcal E = \mathcal E_{up} + \mathcal E_{low} + \mathcal E_{cross} &\leq C(\zeta, \Lambda_0, \mu_0) R^{-2s} \int_{B_R}(f-l + \varepsilon)^{\zeta}(v)\chi_{f > l} \dd v\\
	&\quad- \int_{B_R}\int_{\R^d \setminus B_R}(f-l)_+(w) \eta^2(v) f_{l,\varepsilon}^{-(1-\zeta)}(v) K_f(v, w)\chi_{f(v) > l} \dd w \dd v.
\eal
Note that this estimate holds true for any $(t, x)$. Moreover, since the only positive contribution stems from $\mathcal E_{up}$, the estimate in \eqref{eq:E} is localised in time and space.

\textit{Step 2.: Transport.}

Finally we use the equation. For the transport operator, we first note that since $f-l = (f-l)_+ + (f-l)_-$ and since $f$ conserves mass, there holds
\begin{align*}
	 \iint_{\R^d\times\R^d}&\left( (f-l)_-(T) - (f-l)_-(0)\right) \dd x \dd v \\
	&=  \iint_{\R^d\times\R^d}\left( (f-l)(T) - (f-l)(0)\right) \dd x \dd v -\iint_{\R^d\times\R^d} \left((f-l)_+(T) - (f-l)_+(0) \right)\dd x \dd v\\
	&=-\iint_{\R^d\times\R^d}\left( (f-l)_+(T) - (f-l)_+(0)\right) \dd x \dd v. 
\end{align*}
Now, we distinguish for any $x, v$ the case when $f(0) \geq \min(l, f(T))$, or $f(0) < \min(l, f(T)) \leq \max(l, f(T))$. In the former case, if $f(T) > l$, then $f(0) \geq l$, and the integral vanishes. Else, $f(T) \leq l$, so that $f(0) \geq f(T)$ and $(f-l)_-(T) - (f-l)_-(0) \leq f(T) - l  + l - f(0) \leq 0$. Thus,
\[
	\iint_{\R^d\times\R^d}\left( (f-l)_-(T) - (f-l)_-(0)\right) \dd x \dd v \leq 0. 
\]
In the latter case, when $f(0) \leq \max(l, f(T))$, if $f(T) \geq l$ then $f(0) \leq f(T)$ and $(f-l)_+(0) -f(T) + l\leq f(0) - f(T) \leq 0$ so that 
\begin{align*}
	\iint_{\R^d\times\R^d}\left( (f-l)_-(T) - (f-l)_-(0)\right) \dd x \dd v =-\iint_{\R^d\times\R^d}\left( (f-l)_+(T) - (f-l)_+(0)\right) \dd x \dd v \leq 0.
\end{align*}
Else if $f(T) < l$, then $f(0) \leq l$ and the integral vanishes. Combining this with the divergence theorem, integration by parts and using that $\abs{\mathcal T \eta} \sim R^{-2s}$, we get
\bal\label{eq:transport}
	\int_{[0, T] \times \R^{2d}} \mathcal T f(z)\psi_l(z) f_{l, \varepsilon}^{-(1-\zeta)}(z)\dd z &= \int_{[0, T] \times\R^{2d}} \mathcal T f(z) \eta^2(z)\chi_{f > l}(v) f_{l, \varepsilon}^{-(1-\zeta)}(z)\dd z\\ 
	&\quad + \int_{[0, T] \times\R^{2d}} \mathcal T f(z) \chi_{f < l}(z) \varepsilon^{-(1-\zeta)}\dd z\\
	&=\frac{1}{\zeta} \int_{\R^{1+2d}} \mathcal T ((f-l)_+ + \varepsilon)^{\zeta} (z) \eta^2(z) \dd z\\ 
	&\quad + \int_{[0, T] \times\R^{2d}} \mathcal T (f-l)_-(z)  \varepsilon^{-(1-\zeta)}\dd z\\
	&= - \frac{1}{\zeta} \int_{\R^{1+2d}} ((f-l)_+ + \varepsilon)^{\zeta} (z) \mathcal T \eta^2(z) \dd z\\ 
	&\quad + \int_{\R^{2d}}  \Big((f-l)_-(T, x, v) - (f-l)_-(0, x, v) \Big) \varepsilon^{-(1-\zeta)}\dd x \dd v\\
	&\leq C(\zeta) R^{-2s} \int_{Q_R} ((f-l)_+ + \varepsilon)^{\zeta} (z) \dd z.
\eal
Thus by \eqref{eq:boltzmann-div-l}, \eqref{eq:transport}, and \eqref{eq:E}
\bals
	0 &\leq \int_{[0, T] \times \R^{2d}} \mathcal T f(z)  \psi_l(z) f_{l, \varepsilon}^{-(1-\zeta)}(z)\dd z + \int_{[0, T] \times \R^{2d}}\mathcal E\left(f-l, \psi_l f_{l, \varepsilon}^{-(1-\zeta)}\right) \dd x \dd t\\
	&\leq C(\zeta) R^{-2s}\int_{Q_R}  (f-l+\varepsilon)^{\zeta}(z) \eta^2(z)\chi_{f > l}(z)\dd z\\
	&\quad- \int_{\R^{1+d}}\int_{B_R}\int_{\R^d \setminus B_R}(f-l)_+(w) \eta^2(z) f_{l,\varepsilon}^{-(1-\zeta)}(v) K_f(v, w)\chi_{f(v) > l}  \dd w \dd z,
\eals
or rearranged
\bals
	\int_{Q_{\frac{3R}{4}}}\int_{\R^d \setminus B_R}(f-l)_+(w)  f_{l,\varepsilon}^{-(1-\zeta)}(v) K_f(v, w)\chi_{f(v) > l}  \dd w \dd z\leq C R^{-2s}\int_{Q_R}  (f-l+\varepsilon)^{\zeta}(z) \eta^2(v)\chi_{f > l}(v)\dd z.
\eals

\textit{Step 3.: Conclusion.}

This implies \eqref{eq:L1-tail}, if we let $\varepsilon \to 0$, take out the infimum of $(f-l)^{-(1-\zeta)}$ over $z \in Q_{\frac{3R}{4}} \cap \chi_{f > l}$ on the left hand side, divide the estimate by the infimum (note $f - l \in L^\infty(Q_R)$), and use $(\inf g)^{-1} = \sup g^{-1}$, so that 
\bals
	\int_{Q_{\frac{3R}{4}}}\int_{\R^d \setminus B_R} (f-l)_+(w) \chi_{f > l}(v) K_f(v, w) \dd w\dd v \leq CR^{-2s} \sup_{Q_{\frac{3R}{4}} \cap f > l} (f-l)^{1-\zeta} \int_{Q_R} (f-l)_+^{\zeta} \dd z.
\eals
\end{proof}

\section{Linear $L^1$ to $L^\infty$ bound}

We use a De Giorgi argument in $L^1$ to derive a linear $L^1$-$L^\infty$ bound for solutions to \eqref{eq:boltzmann}. 
\begin{proposition}[$L^1$-$L^\infty$ bound]\label{prop:L1-Linfty}
Let $0 < R \leq1$, and let $f \in L^2([-3, 0]\times B_1; L^\infty(\R^d) + H^s(B_1))$ be a non-negative solution to \eqref{eq:boltzmann-id} in $Q_R$ with a kernel $K_f$ and a coefficient $\Lambda_f$ that satisfy \eqref{eq:cancellation}, \eqref{eq:coercivity}, \eqref{eq:upperbound}, \eqref{eq:1.6}, \eqref{eq:1.7}, \eqref{eq:abs-K-skew}.
Then there exists a large constant $C > 1$ depending on $s, d, \lambda_0, \Lambda_0, \mu_0$ such that 
\beqs
	 \textrm{for a.e. $z_1 \in Q_{\frac{R}{8}}$:} \qquad  f(z_1) \leq  C R^{-(2d(1+s) + 2s)}\int_{Q_R} f(z) \dd z.
\eeqs
\end{proposition}
The proof of this proposition makes use of the gain of integrability stemming from the fractional Kolmogorov equation.
\subsection{Fractional Kolmogorov}
We consider the fractional Kolmogorov equation given by
\beq\label{eq:frac-kolm}
	\partial_t f + v \cdot \nabla_x f + (-\Delta_v)^s f =  h_1 + h_2
\eeq
for some $h_2 \in L^1\big([0, \tau]\times \R^{2d}\big)$ and some $h_1$ such that $\norm{(-\Delta_v)^{-\frac{s+\epsilon}{2}} h_1}_{L^1([0, \tau]\times \R^{2d})} < +\infty$ for $0 \leq \epsilon < s$.
\begin{proposition}\label{prop:goi}
Let $0 \leq f$ solve (or be a sub-solution of) \eqref{eq:frac-kolm} in $[0, \tau]\times \R^{2d}$, with $f(0, x, v) = f_0(x, v) \in L^1\cap L^2(\R^{2d})$, and with $h = h_1 + h_2$ where $h_1, h_2\in L^1\cap L^2([0, \tau]\times\R^{2d})$ such that 
$$\norm{(-\Delta_v)^{-\frac{s+\epsilon}{2}} h_1}_{L^1([0, \tau]\times \R^{2d})} < +\infty.$$
Then for any $0 \leq \epsilon < s$ and any $1 \leq p < 1+\frac{s-\epsilon}{s + 2d(s+1) + \epsilon}$ there holds
\beqs
	\norm{f}_{L^p([0, \tau]\times \R^{2d})} \lesssim \tau^{\frac{1}{p}-\alpha_\epsilon}\norm{(-\Delta_v)^{-\frac{s+\epsilon}{2}} h_1}_{L^1([0, \tau]\times \R^{2d})} +\tau^{\frac{1}{p}+\frac{1}{2}-\alpha_0}\norm{h_2}_{L^1([0, \tau]\times \R^{2d})} + \tau^{\frac{1}{p}+\frac{1}{2}-\alpha_0}\norm{f_0}_{L^1(\R^{2d})},
\eeqs
where $\alpha_\epsilon = d\left(1+\frac{1}{s}\right)\left(1-\frac{1}{p}\right)+\frac{s+\epsilon}{2s}$.
\end{proposition}
We refer the reader to \cite[Proposition 4.1]{AL-div}.

\subsection{Proof of Proposition \ref{prop:L1-Linfty}}
We use the tail bound from Proposition \ref{prop:tail-bound} to get a local energy estimate. Then we use the gain of integrability in $L^1$ from Proposition \ref{prop:goi}, by comparing the solution of \eqref{eq:boltzmann-id} to the fractional Kolmogorov equation \eqref{eq:frac-kolm}. We bound the right hand side by the local energy estimate, and the tail term again with Proposition \ref{prop:tail-bound}. Finally, we iterate the so gained local a priori estimate on level set functions.

\textit{Step 1: Energy estimate.}
 Let $\phi \in C_c^\infty(\R^{1+2d})$ be such that $\phi(t, x, v) = 0$ for $(t, x, v)$ outside $Q_{\frac{R}{4}}$ and $\phi(t, x, v) = 1$ for $(t, x, v) \in Q_{\frac{R}{8}}$. We then test \eqref{eq:boltzmann-id} with $(f - l)_+\phi^2$, so that if we denote by $$\mathcal L f(v) := \int_{\R^d} \big(f(w) - f(v)\big) K_f(v, w) \dd w,$$ we get
\bal\label{eq:aux1-improved-DG}
	\Lambda_f&\int_{\R^{1+2d}}  f(z) \varphi^2(z) (f-l)_+(z) \dd z \\
	&\geq\frac{1}{2} \int_{\R^{1+2d}} \mathcal T (f-l)^2_+(z)  \phi^2(z)  \dd z - \int_{\R^{1+2d}}\mathcal L f(z) (f - l)_+(z)\phi^2(z) \dd z.
\eal
We observe that for any $\frac{R}{4} < r$, if we abbreviate $(f-l)_+ =: {f_l}_+$ and if we denote by $\Omega_{\rho} := [-\rho^{2s}, 0] \times B_{\rho^{1+2s}}$, then
\bal\label{eq:aux-bound-0}
	-\int_{\R^{1+2d}}&\mathcal L f(z) (f - l)_+(z)\phi^2(z) \dd z \\
	&= \int_{\R^{1+2d}}\int_{\R^d} \big[f(t, x, v) - f(t, x, w)\big]\big((f-l)_+\phi^2\big)(t, x, v)K_f(t, x, v, w) \dd w \dd z \\
	&= \frac{1}{2}\int_{\Omega_{\frac{R}{4}}} \int_{B_r} \int_{B_r} \big[f(v) - f(w)\big]\Big[\big({f_l}_+\phi^2\big)(v)- \big({f_l}_+\phi^2\big)(w)\Big]K_f(v, w) \dd w \dd z\\
	&\quad+\frac{1}{2} \int_{\Omega_{\frac{R}{4}}}\int_{B_r}\int_{B_r} \big[f(v) - f(w)\big]\Big[\big({f_l}_+\phi^2\big)(v) +\big({f_l}_+\phi^2\big)(w)\Big]K_f(v, w) \dd w \dd \dd x \dd t \\
	&\quad+\int_{\Omega_{\frac{R}{4}}}\int_{B_r}\int_{\R^d \setminus B_r} \big[f(v) - f(w)\big]\big({f_l}_+\phi^2\big)(z) K_f(v, w)\dd w \dd z.
\eal

\textit{Step 1-i.: Transport.}
We integrate by parts the transport term, and use $\mathcal T \phi \sim - R^{-2s}$.
\bal\label{eq:step1-i-a}
	-\frac{1}{2} \int_{\R^{1+2d}} \mathcal T(f-l)^2_+(z)  \phi^2(z)  \dd z
	&= \frac{1}{2} \int_{\R^{1+2d}}  (f-l)^2_+(t, x, v)  \mathcal T\phi^2(t, x, v)  \dd z \\
	&\leq C R^{-2s} \int_{Q_{\frac{R}{4}}} (f-l)^2_+(z) \dd z.
\eal	

\textit{Step 1-ii.: Tail bound.}
We use Proposition \ref{prop:tail-bound} for $f$, which we assume to be a solution, in particular a super-solution, so that for $r = \frac{R}{2}$, using \eqref{eq:upperbound-2}, 
\bal\label{eq:step1-ii}
	\int_{\Omega_{\frac{R}{4}}} &\int_{B_r}\int_{\R^d \setminus B_r} \big[f(w) - f(v)\big]\big({f_l}_+\phi^2\big)(z) K_f(v, w)\dd w \dd z\\
	&\leq \sup_{Q_{\frac{R}{4}}} (f-l)_+ \int_{Q_{\frac{R}{4}}} \int_{\R^d \setminus B_{\frac{R}{2}}} (f-l)_+(w)  K_f(v, w)\chi_{f > l}(v)\dd w \dd z\\
	&\leq C R^{-2s}\sup_{Q_{\frac{R}{4}}} (f-l)_+\sup_{Q_{\frac{R}{2}}} (f-l)_+^{1-\zeta} \int_{Q_{\frac{R}{2}}}(f-l)^{\zeta}_+(z) \dd z\\
	&\leq C R^{-2s}\left(\sup_{Q_{\frac{R}{2}}} (f-l)_+\right)^{2-\zeta} \int_{Q_{\frac{R}{2}}}(f-l)^{\zeta}_+(z) \dd z.
\eal
We used $-f(v) < -l$ in the first inequality, and Young's inequality in the last inequality.

\textit{Step 1-iii: Not-too-non-local operator.}
What remains to be estimated is the local contribution of the non-local operator, 
\bals
	\mathcal E_{loc} &:= \frac{1}{2} \int_{B_r} \int_{B_r} \big[f(v) - f(w)\big]\Big[\big({f_l}_+\phi^2\big)(v) -\big({f_l}_+\phi^2\big)(w)\Big]K_f(v, w) \dd w \dd v\\
	&+\frac{1}{2} \int_{B_r}\int_{B_r} \big[f(v) - f(w)\big]\Big[\big({f_l}_+\phi^2\big)(v) +\big({f_l}_+\phi^2\big)(w)\Big]K_f(v, w) \dd w \dd v =: \mathcal E_{loc}^{sym} + \mathcal E_{loc}^{skew},
\eals
where we recall ${f_l}_+ = (f-l)_+$.

\textit{Claim 1.} For $z = (t, x, v)$ and $r = \frac{R}{2}$
\bal\label{eq:claim1}
	\mathcal E_{loc}^{sym} &\geq \frac{1}{2}\int_{B_r}\int_{B_r} \Big[\big({f_l}_+\phi\big)(v) -\big({f_l}_+\phi\big)(w) \Big]^2K_f(v, w) \dd v- Cr^{-2s}  \int_{B_r}  {f_l}_+^2(v)\dd v\\
	&\quad +\frac{1}{2}\int_{B_r}\int_{B_r} \big[{f_l}_-(v) - {f_l}_-(w)\big]\Big[\big({f_l}_+\phi^2\big)(v) -\big({f_l}_+\phi^2\big)(w) \Big]K_f(v, w) \dd w \dd v.
\eal
See \cite[Claim (4.5)]{AL-div}.

\textit{Claim 2.}
For any $\delta_0 \in (0, 1)$, there exists a constant $C > 0$ depending on $\delta_0$ such that
\bal\label{eq:claim2}
	\mathcal E_{loc}^{skew}&\geq \frac{1}{2}\int_{B_r}\int_{B_r} \big[{f_l}_-(v) - {f_l}_-(w)\big]\Big[\big({f_l}_+\phi^2\big)(v) +\big({f_l}_+\phi^2\big)(w) \Big]K_f(v, w) \dd w \dd v  \\
	&\quad-C R^{-2s} \int_{B_{\frac{R}{4}}} {f_l}_+^2(v) \dd v -\left(\frac{1}{4} + \delta_0\right) \int_{B_r}\int_{B_r}\Big[\big({f_l}_+\phi\big)(v) - \big({f_l}_+\phi\big)(w)\Big]^2K_f(v, w) \dd w\dd v.
\eal

To prove \eqref{eq:claim2}, we note
\bals
	\big[{f_l}_+(v) - {f_l}_+(w)\big]&\Big[\big({f_l}_+\phi^2\big)(v) +\big({f_l}_+\phi^2\big)(w) \Big] \\
	&= \Big[\big({f_l}_+^2\phi^2\big)(v) -\big({f_l}_+^2\phi^2\big)(w) \Big] +  {f_l}_+(v){f_l}_+(w) \big[\phi^2(w) - \phi^2(v)\big].
\eals
Together with ${f_l}_-(v) - {f_l}_-(w) = f(v) - f(w) - \big({f_l}_+(v) - {f_l}_+(w)\big),$
this implies
\bal\label{eq:aux1-claim2}
	\int_{B_r}&\int_{B_r} \big[f(v) - f(w)\big]\Big[\big({f_l}_+\phi^2\big)(v) +\big({f_l}_+\phi^2\big)(w) \Big]K_f(v, w) \dd w \dd v\\
	&= \int_{B_r}\int_{B_r} \big[{f_l}_-(v) - {f_l}_-(w)\big]\Big[\big({g_l}_+\phi^2\big)(v) +\big({f_l}_+\phi^2\big)(w) \Big]K_f(v, w) \dd w \dd v\\
	&\quad + \int_{B_r}\int_{B_r} \Bigg\{\Big[\big({f_l}_+^2\phi^2\big)(v) -\big({f_l}_+^2\phi^2\big)(w) \Big] + {f_l}_+(v){f_l}_+(w) \big[\phi^2(w) - \phi^2(v)\big]\Bigg\}K_f(v, w) \dd w \dd v.
\eal
Then we use the cancellation \eqref{eq:1.6} 
\bal\label{eq:aux2-claim2}
	 \int_{B_r}\int_{B_r} &\Big[\big({f_l}_+\phi\big)^2(w) -\big({f_l}_+\phi\big)^2(v) \Big]K_f(v, w) \dd w \dd v  \\
	 &= \int_{B_r}\int_{B_r} \big({f_l}_+\phi\big)^2(w)\big[K_f(v, w)-K_f(w, v)\big] \dd w \dd v  \leq \Lambda_0 \int_{B_r} \big({f_l}_+^2\phi^2\big)(w)\dd w.
\eal
Moreover, we bound with Young's inequality, on the one hand, and using $-2ab = (a-b)^2 - a^2 - b^2$, on the other hand,
\bal\label{eq:aux-mixterm}
	{f_l}_+(v){f_l}_+(w) &\big[\phi^2(v) - \phi^2(w)\big] \\
	&\leq \frac{1}{2}\big[ {f_l}^2_+(v) + {f_l}^2_+(w)\big]\phi^2(v) + \frac{1}{2}\big({f_l}_+(v) - {f_l}_+(w)\big)^2 \phi^2(w) - \frac{1}{2}\big[{f_l}_+^2(v) + {f_l}_+^2(w)\big]\phi^2(w)\\
	&=\frac{1}{2}\big[ {f_l}^2_+(v) + {f_l}^2_+(w)\big]\big[\phi^2(v)-\phi^2(w)\big] + \frac{1}{2}\big({f_l}_+(v) - {f_l}_+(w)\big)^2 \phi^2(w).
\eal
Then, using the proof of \eqref{eq:3.5} and Young's inequality, we see for any $\delta_0\in (0, 1)$
\bal\label{eq:aux-mixterm-2}
	\frac{1}{2}\big({f_l}_+(v) - {f_l}_+(w)\big)^2 \phi^2(w) \leq \left(\frac{1}{2} + \delta_0\right) \Big[\big({f_l}_+\phi\big)(v) - \big({f_l}_+\phi\big)(w)\Big]^2 + C(\delta_0)\big(\phi(v) - \phi(w)\big)^2 {f_l}_+^2(v).
\eal
Thus, first using \eqref{eq:aux-mixterm}, second doing a Taylor expansion of $\phi$ and using \eqref{eq:aux-mixterm-2}, third using the symmetry of the integrand and the upper bound \eqref{eq:upperbound}, and finally using the upper bound \eqref{eq:upperbound} again for $s \in (0, 1/2)$, or the cancellation \eqref{eq:1.7} and the upper bound \eqref{eq:upperbound} for $s \in [1/2, 1)$, as well as the definition of $\eta$, we find
\bal\label{eq:aux3-claim2}
	\int_{B_r}\int_{B_r}& {f_l}_+(v){f_l}_+(w) \big[\phi^2(v) - \phi^2(w)\big]K_f(v, w) \dd w \dd v\\
	&\leq \frac{1}{2} \int_{B_r}\int_{B_r}\big[ {f_l}^2_+(v) + {f_l}^2_+(w)\big]\big[\phi^2(v)-\phi^2(w)\big] K_f(v, w)  \dd w\dd v\\
	&\quad +\frac{1}{2} \int_{B_r}\int_{B_r}\big({f_l}_+(v) - {f_l}_+(w)\big)^2 \phi^2(w) K_f(v, w)  \dd w\dd v\\
	&\leq \frac{1}{2} \int_{B_r}\int_{B_r}\big[ {f_l}^2_+(v) D_v \phi^2(v) + {f_l}^2_+(w) D_v\phi^2(w)\big](v-w)K_f(v, w)  \dd w\dd v\\
	&\quad + \norm{D_v^2 \phi^2}_{L^\infty} \int_{B_r}\int_{B_r}\big[ {f_l}^2_+(v) + {f_l}^2_+(w)\big] \abs{v-w}^2K_f(v, w)  \dd w\dd v\\
	&\quad +\left(\frac{1}{2} + \delta_0\right) \int_{B_r}\int_{B_r} \Big[\big({f_l}_+\phi\big)(v) - \big({f_l}_+\phi\big)(w)\Big]^2 K_f(v, w)  \dd w\dd v\\
	&\quad + C(\delta_0)\int_{B_r}\int_{B_r}\big(\phi(v) - \phi(w)\big)^2 {f_l}_+^2(v)K_f(v, w)  \dd w\dd v\\
	&\leq \frac{1}{2} \int_{B_r}\int_{B_r} {f_l}^2_+(v) D_v \phi^2(v) (v-w)\big[K_f(v, w)-K_f(w,v)\big]  \dd w\dd v\\
	&\quad + CR^{2-2s} \norm{D_v^2 \phi^2}_{L^\infty} \int_{B_r}{f_l}^2_+(v) \dd v\\
	&\quad +\left(\frac{1}{2} + \delta_0\right) \int_{B_r}\int_{B_r} \Big[\big({f_l}_+\phi\big)(v) - \big({f_l}_+\phi\big)(w)\Big]^2 K_f(v, w)  \dd w\dd v\\
	&\leq CR^{1-2s}  \norm{D_v\phi^2}_{L^\infty} \int_{B_r}{f_l}^2_+(v)\dd v+CR^{-2s} \int_{B_r}{f_l}^2_+(v) \dd v\\
	&\quad +\left(\frac{1}{2} + \delta_0\right) \int_{B_r}\int_{B_r} \Big[\big({f_l}_+\phi\big)(v) - \big({f_l}_+\phi\big)(w)\Big]^2 K_f(v, w)  \dd w\dd v\\
\eal
Combining \eqref{eq:aux1-claim2} with \eqref{eq:aux2-claim2} and \eqref{eq:aux3-claim2} yields \eqref{eq:claim2}.

Thus, from the claims in \eqref{eq:claim1} and \eqref{eq:claim2}, we infer for $\delta_0$ sufficiently small ($\delta_0 < \frac{1}{8}$) and for $r = \frac{R}{2}$ 
\bal\label{eq:step1-iii}
	\mathcal E_{loc}^{sym}+ \mathcal E_{loc}^{skew} &\geq \frac{1}{2}\int_{B_r}\int_{B_r} \Big[\big({f_l}_+\phi\big)(v) -\big({f_l}_+\phi\big)(w) \Big]^2K_f(v, w) \dd v- Cr^{-2s}  \int_{B_r}  {f_l}_+^2(v)\dd v\\
	&\quad +\frac{1}{2}\int_{B_r}\int_{B_r} \big[{f_l}_-(v) - {f_l}_-(w)\big]\Big[\big({f_l}_+\phi^2\big)(v) -\big({f_l}_+\phi^2\big)(w) \Big]K_f(v, w) \dd w \dd v\\
	&\quad +\frac{1}{2}\int_{B_r}\int_{B_r} \big[{f_l}_-(v) - {f_l}_-(w)\big]\Big[\big({f_l}_+\phi^2\big)(v) +\big({f_l}_+\phi^2\big)(w) \Big]K_f(v, w) \dd w \dd v  \\
	&\quad-C R^{-2s} \int_{B_{\frac{R}{4}}} {f_l}_+^2(v) \dd v \\
	&\quad-\left(\frac{1}{4} + \delta_0\right) \int_{B_r}\int_{B_r}\Big[\big({f_l}_+\phi\big)(v) - \big({f_l}_+\phi\big)(w)\Big]^2K_f(v, w) \dd w\dd v\\
	&\geq \frac{1}{8}\int_{B_r}\int_{B_r} \Big[\big({f_l}_+\phi\big)(v) -\big({f_l}_+\phi\big)(w) \Big]^2K_f(v, w) \dd v- Cr^{-2s}  \int_{B_r}  {f_l}_+^2(v)\dd v.
\eal

\textit{Step 1-iv.: Conclusion.}
Combining \eqref{eq:aux1-improved-DG} with \eqref{eq:step1-i-a}, \eqref{eq:step1-ii}, and \eqref{eq:step1-iii}, we conclude
\bal\label{eq:energy-no-tail}
	\frac{1}{8}&\int_{ \Omega_{\frac{R}{4}}}\iint_{B_{\frac{R}{2}} \times B_{\frac{R}{2}}} \Big[\big((f-l)_+\phi\big)(v) -\big((f-l)_+\phi\big)(w) \Big]^2K_f(v, w) \dd w \dd v \dd x \dd t\\
	&\leq CR^{-2s}\int_{Q_{\frac{R}{4}}} (f-l)_+^2(z)\dd z	 +CR^{-2s} \big(\sup_{Q_{\frac{R}{2}}} f\big)^{2-\zeta} \int_{Q_{\frac{R}{2}}}(f-l)^{\zeta}_+(z) \dd z+ \mu_0\sup_{Q_{\frac{R}{4}}} f \int_{Q_{\frac{R}{4}}} (f-l)_+ \dd z.
\eal
Note that we used \eqref{eq:cancellation} to bound $\Lambda_f$.

\textit{Step 2 \& 3: Gain of Integrability \& De Giorgi iteration.} 
These steps carry over verbatim from \cite[Proof of Prop. 3.1]{AL-div}, if we work with the Boltzmann equation in the formulation \eqref{eq:boltzmann-id}.
Then we have the following gain of integrability 
\bal\label{eq:goi-bound}
	\norm{(f-l)_+  \phi}_{L^p(Q_{\frac{R}{8}})}  &\leq CR^{s+\epsilon} \norm{(f-l)_+(t_0)}_{L^1(Q^{t_0}_{\frac{R}{4}})}+  CR^{-s+\epsilon} \norm{(f-l)_+}_{L^1(Q_{\frac{R}{2}})} \\
	&\quad+ C R^{-s+\epsilon}\abs{\{f > l\} \cap Q_{\frac{R}{2}}}^{\frac{1}{2}} \norm{(f-l)_+}_{L^2(Q_{\frac{R}{2}})} \\	
	&\quad +C R^{-s+\epsilon} \left(\sup_{Q_{R} }(f-l)_+\right)^{\frac{2-\zeta}{2}} \left(\int_{Q_{R}}(f-l)^{\zeta}_+(z) \dd z\right)^{\frac{1}{2}} \abs{\{f > l\} \cap Q_{R}}^{\frac{1}{2}}\\
	&\quad +CR^{-s+\epsilon} \sup_{Q_{R}} (f-l)_+^{1-\zeta} \int_{Q_R} (f-l)_+^{\zeta} \dd z\\
	&\quad +C R^{s+\epsilon} \sup_{Q_{R}} f \abs{\{f > l\}\cap Q_R} +C R^{\epsilon} (\sup_{Q_{R}} f)^{\frac{1}{2}} \abs{\{f > l\}\cap Q_R}^{\frac{1}{2}} \norm{(f-l)_+}_{L^1}^{\frac{1}{2}},
\eal
with which we conclude for almost every $(t, x, v) \in Q_{\frac{R}{8}}$
\beqs
	 f(z) \leq L = \delta(1+R^{2s}) \sup_{Q_R}f +  R^{-\frac{(s-\epsilon)p}{p-1}}2^{\frac{4p^2}{(p-1)^2}}\delta^{-\frac{2-\zeta}{\zeta}\frac{p}{p-1}}\int_{Q_R} f(z) \dd z.
\eeqs
For $R \leq 1$, we absorb the first term on the right hand side with a standard iteration argument, concluding the proof of Proposition \ref{prop:L1-Linfty}.

\subsection{Strong Harnack for Boltzmann}
\begin{proof}[Proof of Theorem \ref{thm:sh-be}]
Let $f$ solve \eqref{eq:boltzmann} in $(0, T) \times \R^d \times \R^d$, and assume \eqref{eq:hydro} holds. Then, in particular, $K_f$ and $\Lambda_f$ given in \eqref{eq:boltzmann_kernel} and \eqref{eq:Lambda-f} satisfy \eqref{eq:cancellation}-\eqref{eq:1.7}. Moreover, due to our notion of solutions in Definition \ref{def:sol}, also \eqref{eq:abs-K-skew} is satisfied. 

Let $\mathcal I^- := (\tau_0, \tau_1)$ and $\mathcal I^+ := (\tau_2, \tau_3)$ be two disjoint compactly contained subsets of $(0, T)$ such that $\tau_2 - \tau_1 \geq r_0^{2s}$ for sufficiently small $r_0 < \frac{1}{6}$. 
As a consequence of Proposition \ref{prop:L1-Linfty}, Young's inequality and the Weak Harnack inequality \cite[Theorem 1.6]{IS}, we obtain for any $\delta \in (0, 1)$ and for $\zeta \in (0, 1)$ from the Weak Harnack inequality,
\bals
	\sup_{\mathcal I^- \times Q_{\frac{r_0}{4}}^t} f &\leq C r_0^{-(2d(1+s) + 2s)} \norm{f}_{L^1(\mathcal I^- \times Q_{2r_0}^t)} \leq \delta \sup_{\mathcal I^- \times Q_{2r_0}^t} f + C(\delta) r_0^{-\frac{(2d(1+s) + 2s)}{\zeta}}\norm{f}_{L^\zeta(\mathcal I^- \times Q_{2r_0}^t)}\\
	&\leq \delta \sup_{\mathcal I^- \times Q_{2r_0}^t} f + C(\delta) r_0^{-\frac{(2d(1+s) + 2s)}{\zeta}}\inf_{\mathcal I^+ \times Q_{\frac{r_0}{4}}^t} f.
\eals
We recall that there is no source term appearing after applying the Weak Harnack inequality, due to the positivity of $Q_2$ \eqref{eq:Q2}, which implies that $f$ is a non-negative super-solution to \eqref{eq:boltzmann-id} with zero source term.
Absorbing the first term on the left hand side with a standard iteration argument concludes the proof of Theorem \ref{thm:sh-be}.
\end{proof}

\section{Brief note on bounds of the fundamental solution}\label{sec:bounds-fundsol}
We end this article with a short remark on how to adapt Aronson's method for non-local hypoelliptic equations in divergence form to the case of more general kernels that satisfy the ellipticity assumptions \eqref{eq:coercivity}-\eqref{eq:abs-K-skew} inspired from the Boltzmann collision kernel. 

\subsection{Results}
As a consequence of Theorem \ref{thm:sh-be}, we derive polynomial upper and exponential lower bounds on the fundamental solution of \eqref{eq:boltzmann-id} with coefficients $K_h$ and $\Lambda_h$ given by \eqref{eq:boltzmann_kernel} and \eqref{eq:Lambda-f}, respectively, for a fixed function $h \geq 0$ that satisfies \eqref{eq:hydro}.
To give sense to the next three theorems, we assume existence of a non-negative measurable function $J$, which is the fundamental solution of \eqref{eq:boltzmann-id} linearised around a fixed function $g$, connecting a given point $(t, x, v) \in \R^{1+2d}$ with $(\tau, y, w) \in \R^{1+2d}$, in the sequel denoted by
\beqs
	J(t, x, v; \tau, y, w) = J(t- \tau, x - y - (t-\tau)w, v-w) =: J\big((\tau, y, w)^{-1} \circ (t, x, v)\big),
\eeqs
where $\circ$ denotes the Galilean translation, that is $(t_0, x_0, v_0) \circ (t, x, v) = (t_0 + t, x_0+x + tv_0, v_0 + v)$, which respects the translation invariance of \eqref{eq:boltzmann}.
Moreover, we assume that the fundamental solution $J$ has the following properties:
\begin{enumerate}[i.]

\item For every $t \in \R_+$ there holds the normalisation
\beq\label{eq:normalisation-J}
	\int_{\R^{2d}} J(t, x, v)  \dd x \dd v = 1.
\eeq

\item There holds $J \geq 0$, and for all $(t, x, v), (\tau, y, w) \in \R_+ \times \R^{2d}$ a form of symmetry
\beq\label{eq:symmetry-J}
J\left((\tau, y, w)^{-1} \circ (t, x, v)\right) = J\left((\tau, x, v)^{-1} \circ (t, y, w)\right).
\eeq

\item For any $0 \leq \tau < \sigma < t < T$ and any $(x, v), (y, w) \in \R^{2d}$ the Chapman-Kolmogorov identity holds
\bal\label{eq:representation-J}
	J(t, x, v; \tau, y, w) = \int_{\R^{2d}} J(t, x, v; \sigma, \varphi, \xi)J(\sigma, \varphi, \xi; \tau, y, w) \dd \varphi \dd \xi.
\eal
\end{enumerate}
Then, we deduce, on the one hand, polynomial upper bounds.
\begin{theorem}[Polynomial upper bounds on the fundamental solution]\label{thm:upperbounds}
Let $x, v, y_0, w_0 \in \R^d$, and $0 \leq \tau_0 < \sigma < T$. Let $J$ be the fundamental solution of \eqref{eq:boltzmann-id} in $[0, T] \times \R^{2d}$ with coefficients $K_h$ and $\Lambda_h$ given by \eqref{eq:boltzmann_kernel} and \eqref{eq:Lambda-f}, respectively, for a fixed function $h \geq 0$ that satisfies \eqref{eq:hydro}. Assume $J$ satisfies \eqref{eq:normalisation-J}, \eqref{eq:symmetry-J} and \eqref{eq:representation-J}. Then there exists $C > 0$ depending on $s, \gamma, d, m_0, M_0, E_0, H_0$ (note that these constants refer to the mass, energy and entropy of $h$), such that
\bal\label{eq:J-up}
	J&(\sigma,x, v; \tau_0, y_0, w_0) \\
	&\leq C (\sigma - \tau_0)^{-\frac{2d(1+s)}{2s}}\left[ 1 +\frac{\max\left\{\abs{v- w_0}^{2s},  \abs{x - y_0 - (\sigma - \tau_0) (v-w_0)}^{\frac{2s}{1+2s}}\right\}}{\sigma - \tau_0}\right]^{-\frac{s}{4s}}.
\eal
\end{theorem}
On the other hand, we derive an exponential lower bound.
\begin{theorem}[Exponential lower bounds on the fundamental solution]\label{thm:lowerbounds}
Let $x, v, y_0, w_0 \in \R^d$, and $0 \leq \tau_0 < \sigma < T$. Let $J$ be the fundamental solution of \eqref{eq:boltzmann-id} in $[0, T] \times \R^{2d}$ with coefficients $K_h$ and $\Lambda_h$ given by \eqref{eq:boltzmann_kernel} and \eqref{eq:Lambda-f}, respectively, for a fixed function $h \geq 0$ that satisfies \eqref{eq:hydro}. Assume $J$ satisfies \eqref{eq:normalisation-J}, \eqref{eq:symmetry-J} and \eqref{eq:representation-J}. Then there exists $C > 0$ depending on $s, \gamma, d, m_0, M_0, E_0, H_0$ (note that these constants refer to the mass, energy and entropy of $h$), such that
\beqs
	J(\sigma, x, v; \tau_0, y_0, w_0) \geq C(\sigma-\tau_0)^{-\frac{2d(1+s)}{2s}}\exp\left\{ - C \left( \frac{\abs{x - y_0 - (\sigma-\tau_0)w_0}^{2s}}{(\sigma-\tau_0)^{1+2s}} + \frac{\abs{v-w}^{2s}}{\sigma-\tau_0} \right)\right\}.
\eeqs
\end{theorem}
\begin{remark}
\begin{enumerate}
\item In \cite{auscher-imbert-niebel-1, auscher-imbert-niebel-2}, they construct a fundamental solution operator for hypoelliptic non-local equations with kernels that are general enough to cover the Boltzmann collision operator.
\item We can also draw a connection to the Gaussian lower bounds for solutions to the Boltzmann equation by Imbert-Mouhot-Silvestre in \cite{IMS}. The authors show that for any $t \in [0, T]$ there exists $a(t), b(t) > 0$ such that any non-negative solution of the Boltzmann equation satisfies
\beqs
	f(t, x, v) \geq a(t) e^{-b(t) \abs{v}^2}.
\eeqs
The same authors establish decay estimates for solutions to \eqref{eq:boltzmann} in \cite{IMSfrench}.
\end{enumerate}
\end{remark}

\subsection{On the polynomial upper bounds}
The proof method remains the same as \cite[Section 6]{AL-div}. The on-diagonal bound \cite[Theorem 6.1]{AL-div} follows by Proposition \ref{prop:L1-Linfty}. The off-diagonal bound \cite[Theorem 6.2]{AL-div} follows if we are able to derive Aronson's bound in \cite[Proposition 6.4]{AL-div} without using the divergence form symmetry of the kernel of the non-local operator. This is contained in Proposition \ref{prop:aronson} below, which is the weak divergence form analogue of \cite[Proposition 6.4]{AL-div}. Once we have established Proposition \ref{prop:aronson} and constructed a decay function $H$ satisfying Aronson's condition \eqref{eq:aronson-condition}, the proof of Theorem \ref{thm:upperbounds} carries over almost verbatim from the proof of \cite[Theorem 1.2]{AL-div}.

\subsubsection{Decay relation}

We aim to define a function that satisfies for some $\rho > 0$
\beq\label{eq:aronson-condition}
	\mathcal T H(v)+ \frac{1}{2} \int_{B_\rho(v)} \big[H(w) - H(v)\big] K_h(w, v) \dd w+ \frac{H(v)}{2} \int_{\R^d}  \left[K_h(v, w)-K_h(w, v)\right] \dd w \leq 0.
\eeq
Then we can derive the following statement. 
\begin{proposition}[Aronson's auxiliary proposition]\label{prop:aronson}
Let $0 < \tau_0 <\sigma < T$ and $0 < \rho$.
Let $f \in L^\infty\big((\tau_0, T) \times \R^{2d}\big)$ solve \eqref{eq:boltzmann-div} in $(\tau_0, T) \times \R^{2d}$ with a non-local operator whose kernel is non-negative and satisfies \eqref{eq:upperbound}, \eqref{eq:1.6}, \eqref{eq:1.7}. Then for every bounded function $H : [\tau_0, \sigma] \times \R^{2d} \to [0, \infty)$ such that $H^{\frac{1}{2}} \in L^2\big( (\tau_0, \sigma) \times \R^{d}; H^s_v(\R^d)\big)$  and $D_v H, D_v^2 H \in L^2((\tau_0, \sigma) \times \R^{d}; L^\infty(\R^d))$,
and, moreover, satisfying \eqref{eq:aronson-condition} in $(\tau_0, \sigma) \times \R^{2d}$,
there exists a constant $C = C(\lambda, \Lambda, s, d)$ such that
\bal\label{eq:step1}
	\sup_{t \in (\tau_0, \sigma)} \int_{\R^{2d}} f^2(t, x, v) H(t, x, v)  \dd x \dd v\leq &\int_{\R^{2d}} f^2(\tau_0, x, v) H(\tau_0, x, v) \dd x \dd v \\
	&+C \rho^{-2s}\norm{H}_{L^\infty([\tau_0, \sigma]\times \R^{2d})} \norm{f}_{L^2([\tau_0, \sigma]\times \R^{2d})}^2.
\eal
\end{proposition}

\begin{proof}[Proof of Proposition \ref{prop:aronson}]

For $R \geq \max\{2, 2\rho+1\}$ we consider $\varphi_R \in C^\infty_c(\R^{2d})$ with $0 \leq \varphi_R \leq 1$ such that $\varphi_R \equiv 1$ for $(x, v) \in B_{(R-1)^{1+2s}} \times B_{R-1}$ and $\varphi_R \equiv 0$ for $(x, v)$ outside $B_{R^{1+2s}} \times B_{R}$, with bounded derivatives and such that $\abs{v \cdot \varphi_R} \sim R^{-2s}$. We test \eqref{eq:boltzmann-div} with $fH\varphi_R^2$ where $H$ satisfies \eqref{eq:aronson-condition} over $[\tau_0, \tau_1]\times \R^{2d}$ for $0 \leq \tau_0 \leq \tau_1 \leq \sigma$ and get
\bals
	\int_{[\tau_0, \tau_1]\times \R^{2d}}\mathcal Tf \Big(fH\varphi_R^2\Big)\dd z = &\int_{[\tau_0, \tau_1]\times \R^{2d}} \int_{\R^d} \big(f(w)- f(v)\big) K_h(v, w)\Big(fH\varphi_R^2\Big)(v) \dd w \dd z\\
	& +\int_{[\tau_0, \tau_1]\times \R^{2d}}f(v)\Big(fH\varphi_R^2\Big)(v) \left(\int_{\R^d}\big(K_h(v, w)-K_h(w, v)\big)\dd w\right) \dd z.
\eals

\textit{Step 1: Transport operator.}

First we integrate by parts the transport operator
\bal
	&\int_{[\tau_0, \tau_1]\times \R^{2d}}\mathcal Tf \Big(fH\varphi_R^2\Big) \dd z \\
	&\quad= \frac{1}{2}\int_{\R^{2d}} f^2 H \varphi_R^2\Big\vert_{t = \tau_0}^{\tau_1} \dd x \dd v- \int_{[\tau_0, \tau_1] \times \R^{2d}}f^2H \varphi_R v\cdot \nabla_x\varphi_R\dd z - \frac{1}{2}\int_{[\tau_0,\tau_1]\times\R^{2d}} f^2\varphi_R^2 \mathcal T H \dd z\\
	&\quad\geq \frac{1}{2}\int_{\R^{2d}} f^2 H \varphi_R^2\Big\vert_{t = \tau_0}^{\tau_1} \dd x\dd  v- CR^{-2s} \int_{[\tau_0, \tau_1] \times \R^{2d}}f^2H \varphi_R \dd z - \frac{1}{2}\int_{[\tau_0,\tau_1]\times\R^{2d}} f^2\varphi_R^2 \mathcal T H \dd z.
\label{eq:aux1-transport}
\eal

\textit{Step 2: Non-local operator.}

Now we deal with the non-local term. 
We write
\bals
	\int_{\R^{d}}\big(\tilde {\mathcal L}f \big)fH\varphi_R^2 \dd v &:=\int_{\R^d} \int_{\R^d} \big(f(w)- f(v)\big) K_h(v, w)\Big(fH\varphi_R^2\Big)(v) \dd w \dd v\\
	&\quad+\int_{\R^d} \big(f^2H\varphi_R^2\big)(v) \left(\int_{\R^d}\big(K_h(v, w)-K_h(w, v)\big)\dd w\right) \dd v\\
	&= \frac{1}{2} \int_{B_{2R}}\int_{B_{2R}} \big[f(w) - f(v)\big] \Big[\big(fH\varphi_R^2\big)(v) - \big(fH\varphi_R^2\big)(w)\Big] K_h(v, w)\dd w\dd v \\
	&\quad+ \frac{1}{2} \int_{B_{2R}}\int_{B_{2R}} \big[f(w) - f(v)\big] \Big[\big(fH\varphi_R^2\big)(v) + \big(fH\varphi_R^2\big)(w)\Big]K_h(v, w)  \dd w\dd v\\
	&\quad+\int_{\R^d} \big(f^2H\varphi_R^2\big)(v) \left(\int_{\R^d}\big(K_h(v, w)-K_h(w, v)\big)\dd w\right) \dd v\\
	&\quad+ \int_{B_{2R}}\int_{\R^d \setminus B_{2R}} \big[f(w) - f(v)\big] \big(fH\varphi_R^2\big)(v) K_h(v, w) \dd w\dd v.
\eals
Then we note that 
\bals
	 \big[f(w) - f(v)\big]& \Big[\big(fH\varphi_R^2\big)(v) + \big(fH\varphi_R^2\big)(w)\Big]\\
	 &=  \big[f(w) +f(v)\big] \Big[\big(fH\varphi_R^2\big)(v) - \big(fH\varphi_R^2\big)(w)\Big] - 2 f(v) \big(fH\varphi_R^2\big)(v)  + 2 f(w) \big(fH\varphi_R^2\big)(w),
\eals
and moreover, by Young's inequality
\bals
	f(w) &\Big[\big(fH\varphi_R^2\big)(v) - \big(fH\varphi_R^2\big)(w)\Big]  \\
	&\leq \frac{1}{2}f^2(w)\Big[ \big(H\varphi_R^2\big)(v)  - \big(H\varphi_R^2\big)(w)\Big] + \frac{1}{2} f^2(v) \big(H\varphi_R^2\big)(v) - \frac{1}{2} f^2(w)\big(H\varphi_R^2\big)(w).
\eals
Thus
\bal\label{eq:nonloc-split}
	\int_{\R^{d}}\big(\tilde {\mathcal L}f \big)fH\varphi_R^2 \dd v 
	&\leq   \frac{1}{2} \int_{B_{2R}}\int_{B_{2R}} f^2(w)  \Big[\big(H\varphi_R^2\big)(v) - \big(H\varphi_R^2\big)(w)\Big] K_h(v, w)\dd w\dd v \\
	&\quad + \frac{1}{2} \int_{B_{2R}}\int_{B_{2R}}  \Big[\big(f^2H\varphi_R^2\big)(w) - \big(f^2H\varphi_R^2\big)(v)\Big] K_h(v, w)\dd w\dd v\\
	&\quad +\int_{\R^d} \big(f^2H\varphi_R^2\big)(v) \left(\int_{\R^d}\big(K_h(v, w)-K_h(w, v)\big)\dd w\right) \dd v\\
	&\quad+ \int_{B_{2R}}\int_{\R^d \setminus B_{2R}} \big[f(w) - f(v)\big] \big(fH\varphi_R^2\big)(v) K_h(v, w) \dd w\dd v\\
	&\leq   \frac{1}{2} \int_{B_{2R}}\int_{B_{2R}} f^2(v)  \Big[\big(H\varphi_R^2\big)(w) - \big(H\varphi_R^2\big)(v)\Big] K_h(w, v)\dd w\dd v \\
	&\quad + \frac{1}{2} \int_{B_{2R}}\int_{B_{2R}}  \big(f^2H\varphi_R^2\big)(v) \big[K_h(w, v)-K_h(v, w)\big]\dd w\dd v\\
	&\quad +\int_{\R^d} \big(f^2H\varphi_R^2\big)(v) \left(\int_{\R^d}\big(K_h(v, w)-K_h(w, v)\big)\dd w\right) \dd v\\
	&\quad+ \int_{B_{2R}}\int_{\R^d \setminus B_{2R}} \big[f(w) - f(v)\big] \big(fH\varphi_R^2\big)(v) K_h(v, w) \dd w\dd v\\
	&=: \mathcal I_{loc}^H + \mathcal I_{loc}^{skew} +\mathcal I_{nonloc}^{skew} + \mathcal I_{tail}.
\eal

First we note that the last integral in \eqref{eq:nonloc-split} tends to zero as $R \to \infty$:
\bal\label{eq:tail-aux}
	\mathcal I_{tail} &= \int_{B_{2R}}\int_{\R^d \setminus B_{2R}} \big[f(w) - f(v)\big] \big(fH\varphi_R^2\big)(v) K_h(v, w) \dd w\dd v \\
	&\leq C\Lambda R^{-2s} \norm{H\varphi_R^2}_{L^\infty} \norm{f}_{L^1(B_R)} \norm{f}_{L^\infty(\R^d\setminus B_{2R})} \xrightarrow[R\to \infty]{} 0.
\eal

Second, we distinguish the singular from the non-singular part to bound $\mathcal I_{loc}^H$. We have
\bal\label{eq:I1-H}
	\mathcal I_{loc}^H &=  \frac{1}{2} \int_{B_{2R}}\int_{B_{2R}} f^2(v)  \Big[\big(H\varphi_R^2\big)(w) - \big(H\varphi_R^2\big)(v)\Big] K_h(w, v)\dd w\dd v \\
	&=  \frac{1}{2} \int_{B_{2R}}\int_{B_{2R} \cap B_\rho(v) } \dots \dd w \dd v +  \frac{1}{2} \int_{B_{2R}}\int_{B_{2R} \setminus B_\rho(v) } \dots \dd w \dd v=: \mathcal I_{loc}^{H, s} +  \mathcal I_{loc}^{H, ns}. 
\eal
Then we bound the non-singular part using \eqref{eq:upperbound-2}
\bal\label{eq:I1.2}
	 \mathcal I_{loc}^{H, ns} &=  \frac{1}{2} \int_{B_{2R}}\int_{B_{2R} \setminus B_\rho(v) }  f^2(v)  \Big[\big(H\varphi_R^2\big)(w) - \big(H\varphi_R^2\big)(v)\Big] K_h(w, v)\dd w\dd v\\
	&\leq C \rho^{-2s} \norm{H\varphi_R}_{L^\infty(B_{2R})}\norm{f}_{L^2(B_{2R})}^2. 
\eal
For the singular part, we observe that $\varphi_R(v) = \varphi_R(w)$ for any $w \in B_\rho(v)$ and $v \in B_{\frac{R-1}{2}}$ if $R \geq 1 + 2\rho$, so that we further split
\bal\label{eq:I1.1}
	 \mathcal I_{loc}^{H, s} &=  \frac{1}{2} \int_{B_{2R}}\int_{B_{2R} \cap B_\rho(v)}  f^2(v)  \Big[\big(H\varphi_R^2\big)(w) - \big(H\varphi_R^2\big)(v)\Big] K_h(w, v)\dd w\dd v\\
	&= \frac{1}{2} \int_{B_{\frac{R-1}{2}}}\int_{B_{2R} \cap B_\rho(v)}  f^2(v) \varphi_R^2(v) \Big[H(w) - H(v)\Big] K_h(w, v)\dd w\dd v\\
	&\quad +\frac{1}{2} \int_{B_{2R} \setminus B_{\frac{R-1}{2}}}\int_{B_{2R} \cap B_\rho(v)}  f^2(v)  \Big[\big(H\varphi_R^2\big)(w) - \big(H\varphi_R^2\big)(v)\Big] K_h(w, v)\dd w\dd v.
\eal
Then note that as $R \to \infty$, the last integral tends to zero as the integrand is bounded: Taylor expanding $H\varphi_R$, using \eqref{eq:upperbound} and \eqref{eq:1.7}
\bals
	\frac{1}{2} &\int_{B_{2R} \setminus B_{\frac{R-1}{2}}}\int_{B_{2R} \cap B_\rho(v)}  f^2(v)  \Big[\big(H\varphi_R^2\big)(w) - \big(H\varphi_R^2\big)(v)\Big] K_h(w, v)\dd w\dd v\\
	&=\frac{1}{2} \int_{B_{2R} \setminus B_{\frac{R-1}{2}}}\int_{B_{2R} \cap B_\rho(v)}  f^2(v)  D_v \big(H\varphi_R^2\big)(v)\cdot\big(w-v\big)  K_h(w, v)\dd w\dd v\\
	&\quad +C\norm{D_v^2\big(H\varphi_R^2\big)}_{L^\infty(B_{2R} \setminus B_{\frac{R-1}{2}})} \int_{B_{2R} \setminus B_{\frac{R-1}{2}}}\int_{B_{2R} \cap B_\rho(v)}  f^2(v) \abs{w-v}^2 K_h(w, v)\dd w\dd v\\
	&\leq C \rho^{1-2s} \int_{B_{2R} \setminus B_{\frac{R-1}{2}}} f^2(v)  D_v \big(H\varphi_R^2\big)(v)\dd v +C\rho^{2-2s}\norm{D_v^2\big(H\varphi_R^2\big)}_{L^\infty(B_{2R} \setminus B_{\frac{R-1}{2}})} \int_{B_{2R} \setminus B_{\frac{R-1}{2}}}f^2(v) \dd v\\
	&\xrightarrow[R\to \infty]{} 0.
\eals

We combine \eqref{eq:I1-H}, \eqref{eq:I1.2}, \eqref{eq:I1.1} and let $R \to \infty$, so that
\beq\label{eq:I1-H-final}
	\mathcal I_{loc}^H\leq C \rho^{-2s} \norm{H}_{L^\infty(\R^d)}\norm{f}_{L^2(\R^d)} + \frac{1}{2} \int_{\R^d} \int_{B_\rho(v)} f^2(v) \Big[H(w) - H(v)\Big] K_h(w, v)\dd w\dd v. 
\eeq

Thus, by letting $R \to \infty$, we conclude from \eqref{eq:nonloc-split}, \eqref{eq:tail-aux}, \eqref{eq:I1-H-final}, 
\bal\label{eq:nonloc-op}
	\int_{\R^{d}}\big(\tilde {\mathcal L}f \big)fH \dd v &\leq \frac{1}{2} \int_{\R^d} f^2(v) \int_{B_\rho(v)}  \Big[H(w) - H(v)\Big] K_h(w, v)\dd w\dd v \\
	&+ \frac{1}{2} \int_{\R^d} (f^2H)(v) \int_{\R^d } \big[ K_h(v, w) - K_h(w, v)\big]\dd w\dd v + C \rho^{-2s}\norm{H}_{L^\infty(\R^d)} \norm{f}_{L^2(\R^d)}^2.
\eal

\textit{Step 3: Conclusion.}

Equation \eqref{eq:aux1-transport} implies
\bals
	\frac{1}{2}\int_{\R^{2d}} &f^2(\tau_1, x, v) H(\tau_1, x, v) \varphi_R^2 \dd x \dd v \\
	&\leq  \frac{1}{2}\int_{\R^{2d}} f^2(\tau_0, x, v) H(\tau_0, x, v) \varphi_R^2 \dd x \dd v+ CR^{-2s} \int_{[\tau_0, \tau_1] \times \R^{2d}}f^2(z)H(z) \varphi_R(x, v) \dd z \\
	&\quad+ \int_{[\tau_0,\tau_1]\times\R^{2d}} f^2(z) \varphi_R^2(x, v)\mathcal T H(z) \dd z+ \int_{[\tau_0,\tau_1]\times\R^{2d}} \big(\tilde {\mathcal L} f\big)(z) f(z) H(z) \varphi_R^2(x, v) \dd z,
\eals
which as $R \to \infty$ yields together with \eqref{eq:nonloc-op}
\bals
	\frac{1}{2}&\int_{\R^{2d}} f^2(\tau_1, x, v) H(\tau_1, x, v)  \dd x \dd v \\
	&\leq   \frac{1}{2}\int_{\R^{2d}} f^2(\tau_0, x, v) H(\tau_0, x, v) \dd x \dd v +  \int_{[\tau_0,\tau_1]\times\R^{2d}} f^2(z)  \mathcal T H(z) \dd z\\
	&\quad +\frac{1}{2}\int_{[\tau_0,\tau_1]\times\R^{2d}} f^2(z) \int_{B_\rho(v)} \Big[H(w) - H(v)\Big] K_h(w, v) \dd w\dd z\\
	&\quad+ \frac{1}{2}\int_{[\tau_0,\tau_1]\times\R^{2d}} f^2(z) H(z) \int_{\R^d} \big[K_h(v, w)-K_h(w, v)\big]\dd w \dd v +C\rho^{-2s}\norm{H}_{L^\infty} \norm{f}_{L^2([\tau_0, \tau_1]\times \R^{2d})}^2 \\
	&\leq   \frac{1}{2}\int_{\R^{2d}} f^2(\tau_0, x, v) H(\tau_0, x, v) \dd x \dd v +C \rho^{-2s}\norm{H}_{L^\infty} \norm{f}_{L^2([\tau_0, \tau_1]\times \R^{2d})}^2,
\eals
since by construction $H$ satisfies \eqref{eq:aronson-condition}. This concludes the proof of Proposition \ref{prop:aronson}.
\end{proof}

It now only remains to show that we can construct a function $H$ satisfying \eqref{eq:aronson-condition}. 
\begin{lemma}\label{lem:existence-H}
Let $y_0, w_0 \in \R^d$.  Let $0 < \rho$ and $0 \leq \tau_0  <\sigma$. Let $k \geq 1$ and $\alpha \geq 0$ be such that $\sigma - \tau_0 \leq \frac{\rho^{2s}}{4k}$.
For $(t, x, v) \in [\tau_0, \sigma] \times \R^{2d}$, we define $\delta(t) := 2(\sigma - \tau_0) - (t - \tau_0)$ and
\beq\label{eq:H}
	H(t, x, v) := e^{-\max\Big\{1, \frac{1}{3\rho}\max\Big(\abs{v-w_0}, \abs{x - y_0 - (\sigma + t - 2\tau_0)w_0}^{\frac{1}{1+2s}}\Big)\Big\}\log\big(\frac{\rho^{2s}}{k \delta(t)}\big)} e^{\alpha\frac{(\sigma - t)}{\rho^{2s}}}.
\eeq
Then there exist constants $C_1, C_2 > 0$ depending only on $s, d, \Lambda_0$ such that, if $k > C_1$ and $\alpha > C_2$, then $H$ satisfies \eqref{eq:aronson-condition}, where $K_h$ is a non-negative kernel satisfying \eqref{eq:upperbound}, \eqref{eq:1.6}, \eqref{eq:1.7}. 
\end{lemma}
To check that $H$ satisfies \eqref{eq:aronson-condition}, we note that due to \eqref{eq:1.6}, we choose $\alpha$ such that 
\beq\label{eq:a-c-aux}
	-\frac{ \alpha}{\rho^{2s}} H(v) + \frac{1}{2} H(v) \int_{\R^d}  \left[K_h(v, w)-K_h(w, v)\right] \dd w \leq  - \alpha H + \Lambda_0  H \leq 0.
\eeq
Then, for ease of notation, we write
\beqs
	H(t, x, v) = H_0(t, x, v) e^{\alpha\frac{(\sigma - t)}{\rho^{2s}}},
\eeqs
so that if $H_0$ satisfies 
\bal\label{eq:a-c-aux2}
  \mathcal T H_0 &+ \frac{1}{2} \int_{B_\rho(v)} \Big[H_0(w) - H_0(v)\Big] K_h(w, v) \dd w \leq 0,
\eal
then in particular $H$ satisfies \eqref{eq:aronson-condition} due to \eqref{eq:a-c-aux}. One can verify \eqref{eq:a-c-aux2} by a case distinction, similar to \cite[Lemma 6.5]{AL-div}.

Note that the remainder of the proof of Theorem 6.2 and Theorem 6.3 in \cite{AL-div} carries over almost verbatim, upon replacing Proposition 6.4 and Lemma 6.5 of \cite{AL-div} by Proposition \ref{prop:aronson}, and Lemma \ref{lem:existence-H}. In particular, the Theorem \ref{thm:upperbounds} follows similarly to \cite[Theorem 1.2]{AL-div}.

\subsection{On the exponential lower bound}
The only part where we use the divergence form symmetry of the non-local operator in the derivation of the exponential lower bound on the fundamental solution is in \cite[Section 7.3]{AL-div}. Instead, we argue as follows.

Fix $(\tau_0, y_0, w_0) \in [0, T]\times \R^{2d}$ and let $0 < \tau_0 < \sigma < T$. Consider for $\tau_1 < \sigma$ and for some $\alpha$ large enough the function
\beqs
	f(\tau_1, y_1, w_1) = M \int_{\max\left\{ \abs{x - y_0 - (\sigma-\tau_0)w_0}^{\frac{2s}{1+2s}}, \abs{v - w_0}^{2s} \right\} < \alpha(\sigma- \tau_0)} J(\sigma, x, v; \tau_1, y_1, w_1) \dd x \dd v,
\eeqs
where $M \geq 1$ is some constant such that $f(\tau_1, y_1, w_1) \leq M$ for $(\tau_1, y_1, w_1) \in [0, \sigma] \times \R^{2d}$. 
Define $g(\tau_1, y_1, w_1) := f(\tau_1, y_1, w_1) e^{-c_0(\tau_1 - \sigma)}$ for $\tau_1 < \sigma$. Then $g$ is a super-solution to the adjoint of \eqref{eq:boltzmann-div} in $(0, \sigma) \times \R^{2d}$, with initial values
\bals
	\begin{cases}
		g(\sigma, y_1, w_1) = M, \qquad &\textrm{if } \max\left\{ \abs{y_1 - y_0 - (\sigma-\tau_0)w_0}^{\frac{2s}{1+2s}}, \abs{w_1 - w_0}^{2s}\right\} < \alpha(\sigma - \tau_0), \\
		g(\sigma, y_1, w_1) = 0, \qquad &\textrm{if } \max\left\{ \abs{y_1 - y_0 - (\sigma-\tau_0)w_0}^{\frac{2s}{1+2s}}, \abs{w_1 - w_0}^{2s}\right\} > \alpha(\sigma - \tau_0).
	\end{cases}
\eals	
since 
\beqs
	\mathcal L^* g(v) = \int_{\R^d} \big(g(w) -g(v)\big)K_h(v, w) \dd w + \Bigg(\int_{\R^d} K_h(w, v) - K_h(v, w) \dd w \Bigg) g(v), 
\eeqs
so that using \eqref{eq:1.6} and choosing $c_0 > 0$ large enough, we get 
\beqs
	\mathcal T^* g(v)- \mathcal L^* g(v)  =- e^{-c_0(\tau_1 - \sigma)} \mathcal T f - e^{-c_0(\tau_1 - \sigma)} \mathcal L f -\Lambda_0 g +c_0 g \geq 0.
\eeqs
If we set 
\bals
	\tilde g(\tau_1, y_1, w_1) =
	\begin{cases}
		g(\tau_1, y_1, w_1), \qquad &\textrm{if } \tau_1 < \sigma, \\
		M e^{-c_0(\tau_1 - \sigma)}, \qquad &\textrm{if } \tau_1 > \sigma,
	\end{cases}
\eals
then $\tilde g$ is a non-negative super-solution of the adjoint of \eqref{eq:boltzmann-div} in $(0, \infty) \times \R^{2d}$ since $\tilde g \leq M e^{-c_0(\tau_1 - \sigma)}$ for $(\tau_1, y_1, w_1) \in [0, \infty) \times \R^{2d}$, so that for $\tau_1 > \sigma$
\beqs
	\mathcal T^* \tilde g - \mathcal L^* \tilde g = c_0 M e^{-c_0(\tau_1 - \sigma)} -\mathcal L \tilde g  - \Lambda_0 \tilde g = Me^{-c_0(\tau_1 - \sigma)}(c_0 - \Lambda_0) +\int \left[Me^{-c_0(\tau_1 - \sigma)}- \tilde g(w)\right] K_h(v, w) \dd w \geq 0. 
\eeqs

Thus by the Weak Harnack inequality \cite[Theorem 1.6]{IS}, we get using $-c_0(\tau_1 - \sigma) \geq 0$ for $c_0 \geq 0$ and the fact that $\tilde f(\sigma, y, w) = M$ for some $\delta > 0$ sufficiently small
\bals
	\tilde g(\tau_0, y_0, w_0) &\geq \left(\int_{\sigma-\delta}^{\sigma+\delta} \int_{\max\left\{ \abs{y' - y_0- (\tau'-\tau_0)w_0}^{\frac{2s}{1+2s}}, \abs{w' - w_0}^{2s} \right\} < \alpha(\sigma - \tau_0)}  \tilde g^{\zeta}(\tau', y', w') \dd y' \dd w' \dd \tau'\right)^{\frac{1}{\zeta}}\\
	&\geq \left(\int_{\sigma-\delta}^{\sigma+\delta} \int_{\max\left\{ \abs{y' - y_0 - (\tau'-\tau_0)w_0}^{\frac{2s}{1+2s}}, \abs{w' - w_0}^{2s} \right\} < \alpha(\sigma- \tau_0)} \tilde  f^{\zeta}(\tau', y', w') \dd y' \dd w' \dd \tau'\right)^{\frac{1}{\zeta}}\\	
	&\geq c > 0,
\eals
so that 
\beqs
	f(\tau_0, y_0, w_0) \geq c e^{c_0(\tau_0 - \tau_1)}. 
\eeqs
We conclude the proof of Theorem \ref{thm:lowerbounds} with \cite[Theorem 7.2]{AL-div}.

\bibliographystyle{acm}
\bibliography{atm}

\end{document}